\documentclass[10pt]{article}

\pdfoutput=1

\usepackage[latin1]{inputenc}

\usepackage{graphicx}
\usepackage{hyperref}

\usepackage{mystyle}

\graphicspath{{./images/}}

\title{Exact Support Recovery\\for Sparse Spikes Deconvolution}

\author{
	\begin{tabular}{c}
		Vincent Duval and Gabriel Peyr\'e \\
		CNRS and Universit\'e Paris-Dauphine \\
		{\small  \url{{vincent.duval,gabriel.peyre}@ceremade.dauphine.fr} }
	\end{tabular}
}

\begin{document}

\maketitle

% !TEX root = ../DuvalPeyre-SparseSpikes.tex

\begin{abstract}
	This paper studies sparse spikes deconvolution over the space of measures. We focus on the recovery properties of the support of the measure (i.e. the location of the Dirac masses) using total variation of measures (TV) regularization. This regularization is the natural extension of the $\ell^1$ norm of vectors to the setting of measures.
	We show that support identification is governed by a specific solution of the dual problem (a so-called dual certificate) having minimum $L^2$ norm. Our main result shows that if this certificate is non-degenerate (see the definition below), when the signal-to-noise ratio is large enough TV regularization recovers the exact same number of Diracs. We show that both the locations and the amplitudes of these Diracs converge toward those of the input measure when the noise drops to zero. 
	Moreover the non-degeneracy of this certificate can be checked by computing a so-called vanishing derivative pre-certificate. This proxy can be computed in closed form by solving a linear system. 		
 	Lastly, we draw connections between the support of the recovered measure on a continuous domain and on a discretized grid. We show that when the signal-to-noise level is large enough, and provided the aforementioned dual certificate is non-degenerate, the solution of the discretized problem is supported on pairs of Diracs which are neighbors of the Diracs of the input measure. This gives a precise description of the convergence of the solution of the discretized problem toward the solution of the continuous grid-free problem, as the grid size tends to zero.
\end{abstract}

% !TEX root = ../DuvalPeyre-SparseSpikes.tex

\section{Introduction}

%%%%%%%%%%%%%%%%%%%%%%%%%%%%%%%%%%%%%%%%%%%%%%%%%
\subsection{Sparse Spikes Deconvolution}

Super-resolution is a central problem in imaging science, and loosely speaking corresponds to recovering fine scale details from a possibly noisy input signal or image. This thus encompasses the problems of data interpolation (recovering missing sampling values on a regular grid) and deconvolution (removing acquisition blur). We refer to the review articles~\cite{Park-review,Lindberg-review} and the references therein for an overview of these problems. 

We consider in our article an idealized super-resolution problem, known as sparse spikes deconvolution. It corresponds to recovering 1-D spikes (i.e. both their positions and amplitudes) from blurry and noisy measurements. These measurements are obtained by a convolution of the spikes train against a known kernel. This setup can be seen as an approximation of several imaging devices. A method of choice to perform this recovery is to introduce a sparsity-enforcing prior, among which the most popular is a $\ell^1$-type norm, which favors the emergence of spikes in the solution.

%%%%%%%%%%%%%%%%%%%%%%%%%%%%%%%%%%%%%%%%%%%%%%%%%
\subsection{Previous Works}

%%%
\paragraph{Discrete $\ell^1$ regularization. }

$\ell^1$-type techniques were initially proposed in geophysics~\cite{Claerbout-geophysics,santosa1986linear,Levy-Fullagar-81} to recover the location of density changes in the underground for seismic exploration. They were later studied in depth by David Donoho and co-workers, see for instance~\cite{Donoho-superresol-sparse}. Their popularity in signal processing and statistics can be traced back to the development of the basis pursuit method~\cite{chen1999atomi} for approximation in redundant dictionaries and the Lasso method~\cite{tibshirani1996regre} for statistical estimation. 

The theoretical analysis of the $\ell^1$-regularized deconvolution was initiated by Donoho~\cite{Donoho-superresol-sparse}.  Assessing the performance of discrete $\ell^1$ regularization methods is challenging and requires to take into account both the specific properties of the operator to invert and of the signal that is aimed at being recovered. A popular approach is to assess the recovery of the positions of the non-zero coefficients. This requires to impose a well-conditioning constraint that depends on the signal of interest, as initially introduced by Fuchs~\cite{fuchs2004on-sp}, and studied in the statistics community under the name of ``irrepresentability condition'', see~\cite{Zhao-irrepresentability}. A similar approach is used by Dossal and Mallat in~\cite{DossalMallat} to study the problem of support stability over a discrete grid.

Imposing the exact recovery of the support of the signal to recover might be a too strong assumption. The inverse problem community rather focuses on the $L^2$ recovery error, which typically leads to a linear convergence rate with respect to the noise amplitude. The seminal paper of Grasmair et al.~\cite{Grasmair-cpam} gives a necessary and sufficient condition for such a convergence, which corresponds to the existence of a non-saturating dual certificate (see Section~\ref{sec-preliminaries} for a precise definition of certificates).  This can be understood as an abstract condition, which is often difficult to check on practical problems such as deconvolution. 

Note that the continuous setting adopted in the present paper might be seen as a limit of such discrete problems, and in Section~\ref{sec-discrete}, we relate our results to well-known results on discrete grids. 

Let us also note that, although we focus here on $\ell^1$-based methods, there is a vast literature on various non-linear super-resolution schemes. This includes for instance greedy~\cite{Odendaal-music,Lorenz-omp}, root finding~\cite{Blu-fri,Condat-Cadzow}, matrix pencils~\cite{Demanet-pencil} and compressed sensing~\cite{Fannjiang-cs-unresolved,Duarte-spectral-cs} approaches. 

%%%
\paragraph{Inverse problems regularization with measures.}

Working over a discrete grid makes the mathematical analysis difficult. Following recent proposals~\cite{deCastro-beurling,Bredies-space-measures,Candes-toward,Bhaskar-line-spectral}, we consider here this sparse deconvolution over a continuous domain, i.e. in a grid-free setting. This shift from the traditional discrete domain to a continuous one offers considerable advantages in term of mathematical analysis, allowing for the first time the emergence of almost sharp signal-dependent criteria for stable spikes recovery (see references below).  Note that while the corresponding continuous recovery problem is infinite dimensional in nature, it is possible to find its solution using either provably convergent algorithms~\cite{Bredies-space-measures} or root finding methods for ideal low pass filters~\cite{Candes-toward}.

Inverse problem regularization over the space of measures is now well understood (see for instance~\cite{Scherzer-ssvm,Bredies-space-measures}), and requires to perform variational analysis over a non-reflexive Banach space (as in \cite{Hofmann-measures}), which leads to some mathematical technicalities. We capitalize on these earlier works to build our analysis of the recovery performance.

%%%
\paragraph{Theoretical analysis of deconvolution over the space of measures.}

For deconvolution from ideal low-pass measurements, the ground-breaking paper~\cite{Candes-toward} shows that it is indeed possible to construct a dual certificate by solving a linear system when the input Diracs are well-separated. This work is further refined in~\cite{Candes-superresol-noisy} that studies the robustness to noise. In a series of paper~\cite{Bhaskar-line-spectral,Tang-linea-spectral} the authors study the prediction (i.e. denoising) error using the same dual certificate, but they do not consider the reconstruction error (recovery of the spikes). In our work, we use a different certificate to assess the exact recovery of the spikes when the noise is small enough.

In view of the applications of superresolution, it is crucial to understand the precise location of the recovered Diracs locations when the measurements are noisy. Partial answers to this questions are given in~\cite{Fernandez-Granda-support} and~\cite{Azais-inaccurate}, where it is shown (under different conditions on the signal-to-noise level) that the recovered spikes are clustered tightly around the initial measure's Diracs. In this article, we fully answer the question of the position of the recovered Diracs in the setting where the signal-to-noise ratio is large enough. 

% TODO
% See also for a different, less tight, analysis~\cite{Kahane-XUPS}
%~\cite{Azais-inaccurate}~\cite{deCastro-beurling}

%%%%%%%%%%%%%%
\subsection{Formulation of the Problem and Contributions.}

Let $m_0=\sum_{i=1}^N a_{0,i}\delta_{x_{0,i}}$ be a discrete measure defined on the torus $\TT=\RR/\ZZ$, where $a_0 \in \RR^N$ and $x_0 \in \TT^N$. We assume we are given some low-pass filtered observation $y_0=\Phi m_0 \in L^2(\TT)$. Here $\Phi$ denotes a convolution operator with some kernel $\varphi \in C^2(\TT)$. The observation might be noisy, in which case we are given $y_0+w= \Phi m_0 + w$, with $w\in L^2(\TT)$, instead of $y_0$.

Following~\cite{Candes-toward,deCastro-beurling}, we hope to recover $m_0$ by solving the problem
\begin{align}\tag{$\Pp_0(y_0)$}
	\umin{\Phi m=y_0} \normTV{m}.
\label{intro-constrained}
\end{align}
among all Radon measures, where $\normTV{m}$ refers to the total variation (defined below) of $m$. Note that in our setting, the total variation is the natural extension of the $\ell^1$ norm of finite dimensional vectors to the setting of Radon measures, and it should not be mistaken for the total variation of functions, which is routinely used to recover signals or images.

We may also consider reconstructing $m_0$ by solving the following penalized problem for $\la >0$, also known as the Beurling LASSO (see for instance~\cite{Azais-inaccurate}):
\begin{align*}\tag{$\Pp_\la(y_0)$}
	\umin{m} \frac{1}{2} \norm{\Phi m - y_0}_2^2 + \la \normTV{m}.
\label{intro-initial}
\end{align*}
This is especially useful if the observation is noisy, in which case $y_0$ should be replaced with $y_0+w$.

Four questions immediately arise:
\begin{enumerate}
	\item Does the resolution of~\eqref{intro-constrained} for $y_0=\Phi m_0$ actually recover interesting measures $m_0$?
	\item How close is the solution of~\eqref{intro-initial} to the solution of~\eqref{intro-constrained} when $\la$ is small enough? 
	\item How close is the solution of~$(\Pp_\la(y_0+w))$ to the solution of~\eqref{intro-initial} when both $\la$ and $w/\la$ are small enough? 
	\item What can be said about the above questions when solving~\eqref{intro-initial} with measures supported on a fixed finite grid? 
\end{enumerate}

The first question is addressed in the landmark paper~\cite{Candes-toward} in the case of ideal low-pass filtering: measures $m_0$ whose spikes are separated enough are the unique solution of~\eqref{intro-constrained} (for data $y_0=\Phi m_0$). Several other cases (using observations different from convolutions) are also tackled in~\cite{deCastro-beurling}, particularly in the case of non-negative measures.

The second and third questions receive partial answers in~\cite{Bredies-space-measures,Candes-superresol-noisy,Azais-inaccurate,Fernandez-Granda-support}. In~\cite{Bredies-space-measures} it is shown that if the solution of~\eqref{intro-constrained} is unique, the measures recovered by $(\Pp_\la(y_0+w))$ converge to the solution of~\eqref{intro-constrained} in the sense of the weak-* convergence when $\la \to 0$ and $\frac{\norm{w}_2^2}{\la}\to 0$. In~\cite{Candes-superresol-noisy}, the authors measure the reconstruction error using the $L^2$ norm of a low-pass filtered version of recovered measures. In~\cite{Azais-inaccurate}, error bounds are derived from the amplitudes of the reconstructed measure. In~\cite{Fernandez-Granda-support}, bounds are given in terms of the original measure. However, those works provide little information about the structure of the measures recovered by $(\Pp_\la(y_0+w))$: \textit{are they made of less spikes than $m_0$ or, in the contrary, do they present lots of parasitic spikes? What happens if one compels the spikes to belong to a finite grid?}

The fourth question is of primary importance since most numerical schemes for sparse regularization solve a finite dimensional optimization problem over a fixed discretization grid. Following~\cite{Candes-toward}, one can remark that in the noiseless setting, if $m_0$ is recovered over the continuous domain and if its support is included in the grid, $m_0$ is also guaranteed to be recovered by the discretized problem. But this is of little interest in practice because the noise is likely to impact in a different manner the discrete problem and the input measure might fall outside the grid locations. Dossal and Mallat in~\cite{DossalMallat} study the stability of the position of the Diracs on the grid, which leads to overly pessimistic conclusions because noise typically forces the spikes to translate over the domain. Studying the convergence of the discretized problem toward the continuous one is thus important to obtain a precise description of the discretized solution. To the best of our knowledge, the work of ~\cite{Bhaskar-line-spectral} is the only one to provide some conclusion about this convergence in term of denoising error. No previous work has studied the capability of the discretized problem to estimate in a precise manner the location of the spikes of the input measure.

%%%%%%%%%%%%%%%%%%%%%%%%%%%%%%%%%%%%%%%%%%%%%%%%%%%%%%
\paragraph{Contributions.} 

The present paper studies in detail the structure of the recovered measure. For this purpose, we define the minimal $L^2$-norm certificate. This certificate fully governs the behavior of the regularization when both $\la$ and $\norm{w}_2/\la$ are small. 

Our first contribution is a set of results indicating that the regions of saturation of the certificate (when it reaches $+1$ or $-1$) are approximately stable when $\la$ and $\norm{w}_2/\la$ are small enough. This means that the recovered measures are supported closely to the support of the input measure if the latter is identifiable (solution of the noiseless problem~\eqref{intro-constrained}). 

Our second contribution introduces the \textit{Non Degenerate Source Condition}, which imposes that the second derivative of the minimal-norm certificate does not vanish on the saturation points. Under this condition, we show that for $\la$ and $\norm{w}_2/\la$ small enough, the reconstructed measure has exactly the same number of spikes as the original measure and that their locations and amplitudes converge to those of the original one.

Our third contribution shows that under the \textit{Non Degenerate Source Condition}, the minimal norm certificate can actually be computed in closed form by simply solving a linear system. This in turn also implies that the errors in the amplitudes and locations decay linearly with respect to the noise level. 

Our fourth and last contribution focuses on the regularization over a discrete finite grid, which corresponds to the so-called Lasso or Basis Pursuit Denoising problem. We show that when $\la$ and $\norm{w}_2/\la$ are small enough, and provided that the Non Degenerate Source Condition holds, the discretized solution is located on pairs of Diracs adjacent to the input Diracs location. This gives a precise description of how the solution to the discretized problem converges to the one of the continuous problem when the stepsize of the grid vanishes.

Throughout the paper, the proposed definitions and results are illustrated in the case of the ideal low-pass filter, showing that the assumptions are actually relevant. Note that the code to reproduce the figures of this article is available online\footnote{\url{https://github.com/gpeyre/2013-FOCM-SparseSpikes/}}.
 
%%%%%%%%%%%%%%%%%%%%%%%%%%%%%%%%%%%%%%%%%%%%%%%%%%%%%%
\paragraph{Outline of the paper.}

Section~\ref{sec-preliminaries} defines the framework for the recovery of Radon measures using total variation minimization. We also expose basic results that are used throughout the paper. Section~\ref{sec-noise-robust} is devoted to the main result of the paper: we define the \textit{Non Degenerate Source Condition} and we show that it implies the robustness of the reconstruction using $(\Pp_\la(y_0+w))$. In Section~\ref{sec-vanishing} we show how the specific dual certificate involved in the \textit{Non Degenerate Source Condition} can be computed numerically by solving a linear system. Lastly, Section~\ref{sec-discrete} focuses on the recovery of measures on a discrete grid.

%%%%%%%%%%%%%%%%%%%%%%%%%%%%%%%%%%%%%%%%%%%%%%%%%%%%%%
\subsection{Notations}

For any Radon measure $m$ defined on $\TT$, we denote its support by $\supp(m)$.
If $\supp(m)$ is a finite set (in which case we say that $m$ is a discrete measure) and $m\neq 0$, then
 $m$ is of the form $m= \sum_{i=1}^N a_i \delta_{x_i}$, where $N\in \NN^*$, $a\in \RR^N$, $x\in \TT^N$
 and $a_i\neq 0$ and $x_i\neq x_j$ for all $1\leq i,j\leq N$. In the rest of the paper, we shall write $m=m_{a,x}$
to  hint that $m$ has the above decomposition (implying that $a_i\neq 0$ and $x_i\neq x_j$ for all $1\leq i,j\leq N$). 

We also define the \textit{signed support}:
\begin{align*}
  \ssupp m &= (\supp m_+) \times \{1\}\cup (\supp m_-)\times\{-1\} \subset \TT \times \{+1,-1 \}
\end{align*}
where $m_+$ (resp. $m_-$) denotes the positive (resp. negative) part of $m$.
For a discrete measure $m=m_{a,x}$, 
\begin{align*}
\ssupp m &=\left\{(t,v)\in \TT\times \{+1,-1\}, \ m(\{t \})\neq 0 \mbox{ and } \sign m(\{t \})=v \right\}\\
&= \{(x_i,\sign a_i), \ 1\leq i\leq N\}.
\end{align*}
We shall consider restrictions of measures and functions to subsets of $\TT$. For $m \in \Mm(\TT)$ a discrete measure and $J=\{x_1,\ldots,x_k\} \subset \TT$ a finite set, we define 
\eq{
	\restr{m}{J} = a \in \TT^{k}
	\qwhereq
	\foralls i=1,\ldots,k, \quad a_i = m(\{x_i\}).  
}
For $\eta \in C(\TT)$ a continuous function defined on $\TT$, we define
\eq{
	\restr{\eta}{J} = ( \eta(x_j) )_{j=1}^k \in \TT^{k}.
}

 Given a convolution operator $\Phi$ with kernel $t\mapsto \varphi(-t)$, we define $\Phi_x : \RR^N \rightarrow L^2(\TT)$ (resp. $\Phi_x'$, $\Phi_x''$) by
\begin{align}\label{eq-notation-Phix}
\forall a\in \RR^N,\	\Phi_x(a) &= \Phi( m_{a,x} ) = \sum_{i=1}^N a_i \phi(x_i - \cdot), \\
	\Phi'_x(a) &= (\Phi_x(a))' = \sum_{i=1}^N a_i \phi'(x_i - \cdot), \\
	{\Phi''_x}(a) &= (\Phi_x(a))'' = \sum_{i=1}^N a_i \phi''(x_i - \cdot).
\end{align}
We define 
\begin{align}\label{eq-gammax}
	\Ga_x &= (\Phi_x, \Phi_x') : (u,v) \in \RR^N \times \RR^N \mapsto \Phi_x u + \Phi'_x v \in \Ldeux(\TT), \\
	\Ga_x'&= (\Phi'_x, \Phi''_x): (u,v) \in \RR^N \times \RR^N \mapsto \Phi_x' u + \Phi''_x v \in \Ldeux(\TT). 
\end{align}

Eventually, in order to study small noise regimes, we shall consider domains $D_{\alpha,\la_0}$, for $\alpha>0$, $\la_0>0$, where:
\eql{\label{eq-constr-set}
	D_{\alpha,\la_0}=\enscond{(\la,w)\in \RR_+\times L^2(\TT) }{ 0\leq \la \leq \la_0 \qandq \norm{w}_2\leq \alpha \la  }.
}

% !TEX root = ../DuvalPeyre-SparseSpikes.tex

\section{Preliminaries}
\label{sec-preliminaries}

In this section, we precise the framework and we state the basic results needed in the next sections. We refer to~\cite{brezis1999analyse} for aspects regarding functional analysis and to~\cite{ekeland1976convex} as far as duality in optimization is concerned.

%%%%%%%%%%%%%%%%%%%%%%%%%%%%%%%%%%%%%%%%%%%%%%%%%%%%%%%%%%%%%%%%%%%%%%%
\subsection{Topology of Radon Measures}

Since $\TT$ is compact, the space of Radon measures $\Mm(\TT)$ can be defined as the dual of the space $C(\TT)$ of continuous functions on $\TT$, endowed with the uniform norm. It is naturally a Banach space when endowed with the dual norm (also known as the total variation), defined as
\begin{align}
	\forall m \in \Mm(\TT), \quad
	\normTV{m}= \sup 
		\enscond{ \int \psi \d m }{ \psi\in C(\TT), \normi{\psi} \leq 1 }.
\label{eq-def-tv}
\end{align}
In that case, the dual of $\Mm(\TT)$ is a complicated space, and it is strictly larger than $C(\TT)$ as $C(\TT)$ is not reflexive. 

However, if we endow $\Mm(\TT)$ with its weak-* topology (i.e. the coarsest topology such that the elements of $C(\TT)$ define continuous linear forms on $\Mm(\TT)$), then $\Mm(\TT)$ is a locally convex space whose dual is $C(\TT)$.

In the following, we endow $C(\TT)$ (respectively $\Mm(\TT)$) with its weak (respectively its weak-*) topology so that both have symmetrical roles:  one is the dual of the other, and conversely. Moreover, since $C(\TT)$ is separable, the set 
$\enscond{m \in \Mm(\TT) }{ \normTV{m} \leq 1 }$ endowed with the weak-* topology is metrizable.

Given a function $\phi \in C^{2}(\TT, \RR)$, we define an operator 
$\Phi : \Mm(\TT) \rightarrow \Ldeux(\TT)$ as 
\eq{
	\foralls m \in \Mm(\TT), \quad \Phi(m) : t \mapsto \int_{\TT} \phi(x-t) \d m(x).
}
It can be shown using Fubini's theorem that $\Phi$ is weak-* to weak continuous.
Moreover, its adjoint operator $\Phi^* : \Ldeux(\TT) \rightarrow C(\TT)$ is defined as
\eq{
	\foralls y \in \Ldeux(\TT), \quad 
	\Phi^*(y) : t \mapsto \int_{\TT} \phi(t-x) y(x) \d x.
}

%%%%%%%%%%%%%%%%%%%%%%%%%%%%%%%%%%%%%%%%%%%%%%
\subsection{Subdifferential of the Total Variation}

It is clear from the definition of the total variation in~\eqref{eq-def-tv} that it is convex lower semi-continuous with respect to the weak-* topology. Its subdifferential is defined as 
\begin{align}
 	\partial \normTV{m} = \enscond{\eta\in C(\TT)}{\forall \tilde{m}\in \Mm(\TT), \normTV{\tilde{m}} \geq \normTV{m} + \int \eta \, \d(\tilde{m}-m) },
\end{align}
for any $m\in \Mm(\TT)$ such that $\normTV{m}<+\infty$.

Since the total variation is a sublinear function, its subgradient has a special structure. One may show (see Proposition~\ref{prop-subdifferential} in Appendix~\ref{sec-auxiliary}) that
\begin{align}
 	\partial \normTV{m} = \enscond{\eta\in C(\TT)}{ \normi{\eta} \leq 1 \qandq \int \eta \, \d m =\normTV{m}  }.
\end{align}

In particular, when $m$ is a measure with finite support, i.e. $m=\sum_{i=1}^N a_i \delta_{x_i}$ for some $N\in \NN$, with $(a_i)_{1\leq i\leq N}\in (\RR^*)^N$ and distinct $(x_i)_{1\leq i \leq N}\in \TT^N$ 
\begin{align}
 	\partial \normTV{m} = \enscond{\eta\in C(\TT)}{ \normi{\eta} \leq 1 \;\text{and}\; \foralls i=1,\ldots,N, \; \eta(x_i)=\sign(a_i)  }.
\end{align}

%For brevity, we may write the second condition as $\eta (x)= \sign (a)$ where $x=(x_1,\ldots x_N)$, and $a=(a_1,\ldots a_N)$.

%%%%%%%%%%%%%%%%%%%%%%%%%%%%%%%%%%%%%%%%%%%%%%
\subsection{Primal and Dual Problems}

Given an observation $y_0=\Phi m_0 \in \Ldeux(\TT)$ for some $m_0\in \Mm(\TT)$, we consider reconstructing $m_0$ by solving either the relaxed problem for $\la >0$
\eql{\label{eq-initial-pb}\tag{$\Pp_\la(y_0)$}
	\umin{m \in \Mm(\TT)} \frac{1}{2} \norm{\Phi(m) - y_0}_2^2 + \la \normTV{m},
}
or the constrained problem 
\eql{\label{eq-constrained-pbm}\tag{$\Pp_0(y_0)$}
	\umin{\Phi(m)=y_0} \normTV{m}.
}
If $m_0$ is the unique solution of \eqref{eq-constrained-pbm}, we say that $m_0$ is \textit{identifiable}.

In the case where the observation is noisy (i.e. the observation $y_0$ is replaced with $y_0+w$ for $w\in L^2(\TT)$), we attempt to reconstruct $m_0$ by solving $\Pp_\la(y_0+w)$ for a well-chosen value of $\la>0$.

Existence of solutions for~\eqref{eq-initial-pb} is shown in~\cite{Bredies-space-measures}, and existence of solutions
for~\eqref{eq-constrained-pbm} can be checked using the direct method of the calculus of variations (recall that for~\eqref{eq-constrained-pbm}, we assume that the observation is $y_0=\Phi m_0$).

A straightforward approach to studying the solutions of Problem~\eqref{eq-initial-pb} is then to apply Fermat's rule: 
a discrete measure $m=m_{a,x}=\sum_{i=1}^N a_i\delta_{x_i}$ is a solution of $\eqref{eq-initial-pb}$ if and only if there exists $\eta\in C(\TT)$ such that
\eq{
	\Phi^*(\Phi m -y_0) +\la \eta =0,
}
with $\normi{\eta} \leq 1$ and $\eta(x_i)=\sign(a_i)$ for $1\leq i\leq N$.

Another source of information for the study of Problems~\eqref{eq-initial-pb} and~\eqref{eq-constrained-pbm} is given by their associated dual problems. In the case of the ideal low-pass filter, this approach is also the key to the numerical algorithms 
used in~\cite{Bhaskar-line-spectral,Candes-toward,Azais-inaccurate}: the dual problem can be recast into a finite-dimensional problem.

The Fenchel dual problem to~\eqref{eq-initial-pb} is given by
\eql{\label{eq-initial-dual}\tag{$\Dd_\la(y_0)$}
	\umax{ \normi{\Phi^* p} \leq 1} \dotp{y_0}{p} - \frac{\la}{2}\norm{p}_2^2,
}
which may be reformulated as a projection on a closed convex set (see~\cite{Bredies-space-measures,Azais-inaccurate})
\eql{\label{eq-initial-dualbis}\tag{$\Dd'_\la(y)$}
	\umin{ \normi{\Phi^* p} \leq 1} \norm{\frac{y_0}{\la}-p}_2^2.
}
This formulation immediately yields existence and uniqueness of a solution to~\eqref{eq-initial-dual}.

The dual problem to~\eqref{eq-constrained-pbm} is given by
\eql{\label{eq-constrained-dual}\tag{$\Dd_0(y_0)$}
	\usup{ \normi{\Phi^* p}\leq 1} \dotp{y_0}{p}.
}
 Contrary to~\eqref{eq-initial-dual}, the existence of a solution to~\eqref{eq-constrained-dual} is not always guaranteed, so that in the following (see Definition~\ref{def-ndsc}) we make this assumption. 

Existence is guaranteed when for instance $\Im \Phi^*$ is finite-dimensional (as is the case in the framework of~\cite{Candes-toward}). If a solution to~\eqref{eq-constrained-dual} exists, the unique solution of~\eqref{eq-initial-dual} converges to a certain solution of~\eqref{eq-constrained-dual} for $\la \to 0^+$ as shown in Proposition~\ref{prop-gamma-convergence} below.

%%%%%%%%%%%%%%%%%%%%%%%%%%%%%%%%%%%%%%%%%%%%%%
\subsection{Dual Certificates}
\label{sec-dualcertif}

The strong duality between $(\Pp_\la(y_0))$ and~\eqref{eq-initial-dual} is proved in~\cite[Prop.~2]{Bredies-space-measures} by seeing~\eqref{eq-initial-dualbis} as a predual problem for~\eqref{eq-initial-pb}. As a consequence, both problems have the same value and any solution $m_\la$ of~\eqref{eq-initial-pb} is linked with the unique solution $p_\la$ of~\eqref{eq-initial-dual} by the extremality condition
\begin{align}
	\left\{
		\begin{array}{c}
			\Phi^*p_\la \in \partial\normTV{m_\la}, \\
			-p_\la = \frac{1}{\la}(\Phi m_\la - y_0).
		\end{array}
	\right.
\label{eq-extremal-cdt}
\end{align}
Moreover, given a pair $(m_\la,p_\la)\in \Mm(\TT) \times L^2(\TT)$, if relations~\eqref{eq-extremal-cdt} hold, then $m_\la$ is a solution to Problem~\eqref{eq-initial-pb} and $p_\la$ is the unique solution to Problem~\eqref{eq-initial-dual}.

As for~\eqref{eq-constrained-pbm}, a proof of strong duality is given in Appendix~\ref{sec-auxiliary} (see Proposition~\ref{prop-strong-dual}).
If a solution $p^\star$ to~\eqref{eq-constrained-dual} exists, then it is linked to any solution $m^\star$ of~\eqref{eq-constrained-pbm} by
\begin{align}
	\Phi^* p^\star \in \partial{\normTV{m^\star}},
\label{eq-extremal-constrained}
\end{align}
and similarly, given a pair $(m^\star,p^\star)\in \Mm(\TT)\times L^2(\TT)$, if relation~\eqref{eq-extremal-constrained} hold, then $m^\star$ is a solution to Problem~\eqref{eq-constrained-pbm} and $p^\star$ is a solution to Problem~\eqref{eq-constrained-dual}).

Since finding $\eta = \Phi^* p^\star$ which satisfies~\eqref{eq-extremal-constrained} gives a quick proof
that $m^\star$ is a solution of~\eqref{eq-constrained-pbm}, we call $\eta$ a \textit{dual certificate} for $m^\star$.
We may also use a similar terminology for $\eta_\la=\Phi^* p_\la$ and Problem~\eqref{eq-initial-pb}.

In general, dual certificates for~\eqref{eq-constrained-pbm} are not unique, but we consider in the following definition a specific one, which is crucial for our analysis.

\begin{defn}[Minimal-norm certificate]
	When it exists, the minimal-norm dual certificate associated with~\eqref{eq-constrained-pbm} is defined as $\eta_0=\Phi^* p_0$ where $p_0 \in L^2(\TT)$ is the solution of~\eqref{eq-constrained-dual} with minimal norm, i.e.
\begin{align}
	\eta_0=\Phi^* p_0,&  
	\qwhereq
	p_0=\uargmin{p}
		\enscond{ \norm{p}_2 }{ p \mbox{ is a solution of }\eqref{eq-constrained-dual} }.
\label{eq-min-norm-certif}
\end{align}
\end{defn}
Observe that in the above definition, $p_0$ is well-defined provided there exists a solution to Problem~\eqref{eq-constrained-dual}, since $p_0$ is then the projection of $0$ onto the non-empty closed convex set of solutions.
Moreover, in view of the extremality conditions~\eqref{eq-extremal-constrained}, given any solution $m^\star$ to~\eqref{eq-constrained-pbm}, it may be expressed as 
\begin{align}\label{eq-min-norm-certifbis}
	p_0=\uargmin{ p }
		\enscond{ \norm{ p }_2 }{ \Phi^* p \in \partial{\normTV{m^\star}}  }.
\end{align}

\begin{prop}[Convergence of dual certificates]
Let $p_\la$ be the unique solution of Problem~\eqref{eq-initial-dual}, and $p_0$ be the solution of Problem~\eqref{eq-constrained-dual} with minimal norm defined in~\eqref{eq-min-norm-certif}.
Then 
\begin{align*}
\lim_{\la \to 0^+} p_\la = p_0 \quad \mbox{for the } L^2 \mbox{ strong topology.}
\end{align*}
Moreover the dual certificates $\eta_\la= \Phi^*p_\la$ for Problem~\eqref{eq-initial-pb} converge to the minimal norm certificate $\eta_0 = \Phi^*p_0$. More precisely, 
\begin{align}
	\forall k\in \{0,1,2\}, \quad 
	\lim_{\la \to 0^+} \eta_\la^{(k)} = \eta_0^{(k)},
\end{align}
in the sense of the uniform convergence.
\label{prop-gamma-convergence}
\end{prop}

\begin{proof}
Let $p_\la$ be the unique solution of~\eqref{eq-initial-dual}. By optimality of $p_\la$ (resp. $p_0$) for
\eqref{eq-initial-dual} (resp.~\eqref{eq-constrained-dual})
\begin{align}
	\dotp{y_0}{p_\la} - \la \norm{p_\la}_2^2 	& \geq \dotp{y_0}{p_0} -\la \norm{p_0}_2^2, \label{eq-optim-la}\\
	\dotp{y_0}{p_0}  							&\geq  \dotp{y_0}{p_\la}. \label{eq-optim-0}
\end{align}
As a consequence $\norm{p_0}_2^2\geq \norm{p_\la}_2^2$ for all $\la>0$.

Now, let $(\la_n)_{n\in\NN}$ be any sequence of positive parameters converging to $0$. The sequence $p_{\la_n}$  being bounded in $L^2(\TT)$, we may extract a subsequence (denoted $\la_{n'}$) such that $p_{\la_{n'}}$ weakly converges to some $p^\star \in L^2(\TT)$. Passing to the limit in~\eqref{eq-optim-la}, we get 
	$\dotp{y_0}{p^\star} \geq \dotp{y_0}{p_0}$.
Moreover, $\Phi^*p_{\la_n}$ weakly converges to $\Phi^* p^\star$ in $C(\TT)$, so that 
$\normi{\Phi^* p^\star} \leq \liminf_{n'} \normi{\Phi^*p_{\la_{n'}}} \leq 1$, and $p^\star$
  is therefore a solution of~\eqref{eq-constrained-dual}.

But one has 
\eq{
  \norm{p^\star}_2\leq \liminf_{n'} \norm{p_{\la_{n'}}}_2\leq \norm{p_0}_2,
} 
hence $p^\star=p_0$ and in fact $\lim_{{n'}\to +\infty} \norm{p_{\la_{n'}}}_2=\norm{p_0}_2$. As a consequence, $p_{\la_{n'}}$ converges to $p_0$ for the $L^2(\TT)$ strong topology as well. This being true any sequence $\la_n\to 0^+$, we get the result claimed for $p_\la$: assume by contradiction that there exists $\varepsilon_0>0$ and a sequence $\la_n \searrow 0$ such that $\|p_0-p_{\la_n}\|_2 \geq \varepsilon_0$ for all $n\in\NN$. By the above argument we may extract a subsequence $\la_{n'}$ which converges towards $p_0$, which contradicts $\|p_0-p_{\la_n'}\|_2 \geq \varepsilon_0$. Hence $\lim_{\la\to 0}p_\la=p_0$ strongly in $L^2$.
    
It remains to prove the convergence of the dual certificates.
Observing that $\eta_\la^{(k)}(t)=\int \varphi^{(k)}(t-x)p_\la(x) \d x$, we get    
\begin{align*}
  |\eta_\la^{(k)}(t)-\eta_0^{(k)} (t)|&= \absb{ \int_{\TT} \varphi^{(k)}(t-x)(p_\la-p_0)(x) \d x }\\
  &\leq \sqrt{\int_{\TT} |\varphi^{(k)}(t-x)|^2 \d x} \sqrt{\int_{\TT} |(p_\la -p_0)(x)|^2 \d x}\\
	&\leq C \norm{p_\la - p_0}_2,
\end{align*}
where $C>0$ does not depend on $t$ nor $k$, hence the uniform convergence. 
\end{proof}

%%%%%%%%%%%%%%%%%%%%%%%%%%%%%%%%%%%%%%%%%%%%%%%%%%%%%%%%%%%%%
\subsection{Application to the ideal Low-pass filter}

In this paragraph, we apply the above duality results to the particular case of the Dirichlet kernel, defined as
\begin{align}
	\varphi(t) = \sum_{k=-f_c}^{f_c} e^{2i\pi kt} = \frac{\sin \left((2f_c+1)\pi t\right)}{\sin (\pi t)}.
\end{align}
It is well known that in this case the spaces $\Im \Phi$ and $\Im \Phi^*$ are finite-dimensional,
 being the space of real trigonometric polynomials with degree less than or equal to $f_c$.

We first check that a solution to~\eqref{eq-constrained-dual} always exists. As a
consequence, given any measure $m_0$, the minimal norm certificate is well defined. 

\begin{prop}[Existence of $p_0$]
Let $m_0\in \Mm(\TT)$ and $y_0=\Phi m_0\in L^2(\TT)$. There exists a solution of~\eqref{eq-constrained-dual}. As a consequence, $p_0\in L^2(\TT)$ is well defined.
\end{prop}

\begin{proof}
We rewrite~\eqref{eq-constrained-dual} as 
\eq{
	\usup{\normi{\eta} \leq 1, \eta \in \Im \Phi^*} \dotp{m_0}{\eta}.
}
Let $(\eta_n)_{n\in \NN}$ be any maximizing sequence. Then $(\eta_n)_{n\in \NN}$ is bounded in the finite-dimensional space of trigonometric polynomials with degree $f_c$ or less. We may extract a subsequence converging to $\eta^\star\in C(\TT)$. But $\normi{\eta^\star} \leq 1$ and $\eta^\star\in \Im \Phi^*$,
so that $\eta^\star=\Phi^* p^\star$ for some $p^\star$ solution of~\eqref{eq-constrained-dual}.
\end{proof}

A striking result of~\cite{Candes-toward} is that discrete
 measures are identifiable provided that their support is separated enough, i.e. $\Delta(m_0) \geq \frac{C}{f_c}$ for some $C>0$, where $\Delta(m_0)$ is the so-called minimum separation distance.

\begin{defn}[Minimum separation]
The minimum separation of the support of a discrete measure $m$ is defined as
\eq{
 	\Delta(m)=\inf_{(t,t')\in \supp (m)} |t-t'|,
}
where $|t-t'|$ is the distance on the torus between $t$ and $t'\in \TT$, and we assume $t\neq t'$.
\end{defn}

In~\cite{Candes-toward} it is proved that $C \leq 2$ for complex measures (i.e. of the form $m_{a,x}$ for $a \in \CC^N$ and $x \in \TT^N$) and $C \leq 1.87$ for real measures (i.e. of the form $m_{a,x}$ for $a \in \RR^N$ and $x \in \TT^N$). Extrapolating from numerical simulations on a finite grid, the authors conjecture that for complex measures, one has $C\geq 1$.
 In this section we apply results from Section~\ref{sec-dualcertif} to show that for real measures, necessarily $C\geq \frac{1}{2}$.

We rely on the following theorem, proved by P. Tur\'an~\cite{Turan1946}.

\begin{thm}[Tur\'an]
Let $P(z)$ be a non trivial polynomial of degree $n$ such that $|P(1)|=\max_{|z|=1} |P(z)|$. Then for any root $z_0$ of $P$ on the unit circle, $|\arg (z_0)|\geq \frac{\pi}{n}$. Moreover, if $|\arg (z_0)|= \frac{\pi}{n}$, then $P(z)=c(1+z^n)$ for some $c\in \CC^*$.
 \label{thm-turan}
\end{thm}
From this theorem we derive necessary conditions for measures that can be reconstructed by~\eqref{eq-constrained-pbm}.

\begin{cor}[Non identifiable measures]
There exists a discrete measure $m_0$ with $\Delta(m_0)=\frac{1}{2f_c}$ such that $m_0$ is not a solution of~\eqref{eq-constrained-pbm} for $y_0=\Phi m_0$.
\end{cor}

\begin{proof}
Let $m_0=\delta_{-\frac{1}{2f_c}}+ \delta_0 -\delta_{\frac{1}{2f_c}}$, assume by contradiction that $m$ is a solution of~\eqref{eq-constrained-pbm},
 and let $\eta\in C(\TT)$ be an associated dual certificate (which
  exists since $\Im \Phi^*$ is finite-dimensional).
Then necessarily $\eta(-\frac{1}{2f_c})=\eta(0)=1$ and $\eta(\frac{1}{2f_c})=-1$ and by the intermediate value theorem, there exists $t_0\in (0, \frac{1}{2f_c})$ such that $\eta(t_0)=0$.

Writing $\eta(t)=\sum_{k=-f_c}^{f_c}d_k e^{2i\pi kt}$,
  the polynomial $P(z) = \sum_{k=0}^{2f_c}d_{k-f_c}z^k$ satisfies $P(1)=1=\sup_{|z|=1}|P(z)|=|P(e^{\frac{2i\pi}{2f_c}})|$,
   and $P(e^{2i\pi t_0})=0$.

By Theorem~\ref{thm-turan}, we cannot have $|2\pi t_0 - 0 |<\frac{\pi}{2f_c}$ nor $|2\pi t_0-\frac{2\pi}{2f_c}|<\frac{\pi}{2f_c}$, hence $t_0=\frac{1}{4f_c}$ and $P(z)=c(1+z^{2f_c})$, so that $\eta(t)=\cos (2\pi f_ct)$. But this implies $\eta(-\frac{1}{2f_c})=-1$, which contradicts the optimality of $\eta$.
\end{proof}

In a similar way, we may also deduce the following corollary.

\begin{cor}[Opposite spikes separation]
Let $m^\star\in \Mm(\TT)$ be any discrete measure solution of Problem $\Pp_\la(y_0+w)$ or $\Pp_0(y_0)$
 where $y_0=\Phi m_0$ for any data $m_0\in \Mm(\TT)$ and any noise $w\in L^2(\TT)$. If there
  exists $x^\star_0\in \TT$ (resp. $x^\star_1\in \TT$) such that $m^\star(\{x^\star_0\})>0$
   (resp. $m^\star(\{x^\star_1\})<0$), then $|x^\star_0 - x^\star_1 |\geq \frac{1}{2f_c}$.
\end{cor}

% !TEX root = ../DuvalPeyre-SparseSpikes.tex

\section{Noise Robustness}
\label{sec-noise-robust}

This section is devoted to the study of the behavior of solutions to $\Pp_\la(y_0+w)$ for small values of $\la$ and $\norm{w}$. In order to study such regimes, as already defined in~\eqref{eq-constr-set}, we consider sets of the form
\begin{align*}
	D_{\alpha,\la_0}=\enscond{(\la,w)\in \RR_+\times L^2(\TT) }{ 0\leq \la \leq \la_0 \qandq \norm{w}_2\leq \alpha \la  },
\end{align*}
for $\alpha>0$ and $\la_0>0$.

First, we introduce the notion of extended support of a measure. Then we show that this 
concept governs the structure of solutions at small noise regime. After introducing the 
Non Degenerate Source Condition, we state the main result of the paper, \textit{i.e.} that under this assumption,
the solutions of $\Pp_\la(y_0+w)$ have the same number of spikes as the original measure, and that these
spikes converge smoothly to those of the original measure.

%%%%%%%%%%%%%%%%%%%%%%%%%%%%%%%%%%%%%%%%%%%%%
\subsection{Extended signed support}

Our first step in understanding the behavior of solutions to~$\Pp_\la(y_0+w)$ at low noise
 regime is to introduce the notion of extended signed support.
\begin{defn}[Extended signed support]
Let $m_0\in \Mm(\TT)$ such that there exists a solution to~\eqref{eq-constrained-dual}
 (where as usual $y_0=\Phi m_0$), and let $\eta_0\in C(\TT)$ be the associated minimal norm certificate.

The extended support of $m_0$ is defined as:
\begin{align}
  \ext(m_0) = \enscond{ t\in \TT }{ \eta_0(t)=\pm 1 },
  \end{align}
  and the extended signed support of $m_0$ as:
  \begin{align}
  \exts (m_0) = \enscond{ (t,v)\in \TT\times \{+1,-1\} }{ \eta_0(t)=v  }.
  \end{align}
\end{defn}
Notice that $\ext m_0$ and $\exts m_0$ actually depend on $y_0=\Phi m_0$ rather than on $m_0$ itself.
For any measure $m_0\in \Mm(\TT)$, the (signed) support and the extended (signed) support of $m_0$
 are in general not related. Yet, from the optimality conditions~\eqref{eq-extremal-constrained} we observe:
\begin{prop}
Let $m_0\in \Mm(\TT)$ and $y_0=\Phi m_0$ such that there exists a solution to~\eqref{eq-constrained-dual}.
Then:
\begin{itemize}
  \item $m_0$ is a solution to \eqref{eq-constrained-pbm} if and only if $\ssupp m_0 \subset \exts m_0$.
  \item In any case, if $\Phi_{\ext m_0}$ has full rank, the solution to \eqref{eq-constrained-pbm} is unique.
\end{itemize}
\label{prop-extend-solution}
\end{prop}

Here, following the notation~\eqref{eq-notation-Phix}, we have denoted by $\Phi_{\ext m_0}$ the restriction of $\Phi$ to the space of measures with support in $\ext m_0$.
The link between Proposition~\ref{prop-extend-solution} and the source condition~\cite{Burger_Osher04} is discussed in Section~\ref{sec-source-cdt}

%%%%%%%%%%%%%%%%%%%%%%%%%%%%%%%%%%%%%%%%%%%%%%%%%%%
\subsection{Local behavior of the support}

In this paragraph, we focus on the local properties of the support of solutions to $\Pp_\la(y_0+w)$ at low noise regime. As usual, we denote $y_0=\Phi m_0$ for some $m_0\in \Mm(\TT)$. For now, we make as few assumptions as possible on $m_0$. In particular, we do not assume that $\Phi_{\ext m_0}$ has full rank. Any solution to $\Pp_\la(y_0+w)$ (which is not necessarily unique) is denoted by $\tilde{m}_\la$.

\begin{lem}
Assume that there exists a solution to~\eqref{eq-constrained-dual} and let $\epsilon>0$.
Then there exists $\alpha>0$, $\la_0>0$ such that for all $(\la, w)\in D_{\alpha, \la_0}$, 
  \begin{align}
    \ssupp \tilde{m}_\la \subset \left(\exts m_0\right) \oplus\left((-\epsilon,+\epsilon)\times\{0\}\right),
 \end{align}
 where given two sets $A$ and $B$, $A\oplus B=\enscond{a+b }{ a\in A, b\in B}$ denotes their Minkowski sum.
  \label{lem-robust-boites}
\end{lem}

In particular, if $\ext m_0$ consists in isolated points $x_{0,1}, \ldots x_{0,N}$, Lemma~\ref{lem-robust-boites} states that  all the mass of $\tilde{m}_\la$ is concentrated in boxes $(x_{i,0}-\epsilon, x_{i,0}+\epsilon)$, where $\epsilon \rightarrow 0$ when $\la, \norm{w} \rightarrow 0$. Moreover, in each box, $\tilde{m}_\la$ has the sign of $\eta_0(x_{0,i})$.

Also, if $\exts m_0=\emptyset$ (i.e. $y=0$), we see that $\tilde{m}_\la=0$ for $\la$ and $\frac{\norm{w}_2}{\la}$ small enough (in fact,
any $\la_0>0$ and $\alpha = \frac{1}{ \norm{\Phi^*}_{2,\infty}}$ suffices, as can be seen from \eqref{eq-extremal-cdt}).

\begin{proof}
We split the proof in several parts.
%%%%%%%%%%%%%%%%
  \paragraph{Behavior of the minimal norm certificate.}
Let us consider the sets:
\begin{align*}
  \ext^+ = \enscond{ t\in \TT }{ \eta_0(t)=1 },\quad & \ext^- = \enscond{ t\in \TT }{ \eta_0(t)=-1 },\\
  \ext^{+,\epsilon} = \ext^+ \oplus (-\epsilon,\epsilon),\quad &  \ext^{-,\epsilon} = \ext^- \oplus (-\epsilon,\epsilon).
\end{align*}
From the uniform continuity of $\eta_0$, for $\epsilon$ small enough, $\eta_0>\frac{1}{2}$ in $\ext^{+,\epsilon}$ and $\eta_0<-\frac{1}{2}$
in $\ext^{-,\epsilon}$, so that $\ext^{+,\epsilon} \cap \ext^{-,\epsilon}=\emptyset$.

If $\ext^{+,\epsilon} \cup  \ext^{-,\epsilon} \subsetneq \TT$, the set $K_\epsilon = \TT \setminus \left(\ext^{+,\epsilon} \cup  \ext^{-,\epsilon}\right)$ being compact, 
$\sup_{K_\epsilon} |\eta_0|<1$. We define $r= 1 -\sup_{K_\epsilon}|\eta_0|$.

If $\ext^{+,\epsilon} \cup  \ext^{-,\epsilon} = \TT$, the connectedness of $\TT$
implies that $\ext^{+,\epsilon} =\TT$ and $\ext^{-,\epsilon}=\emptyset$, or conversely.
In that case we define $r=1$.

In any case, we see that for all $g\in C(\TT)$, if $\norm{g-\eta_0}_\infty < r$, then
\begin{align}
  \enscond{t\in \TT }{ g(t)=1 }\subset \ext^{+,\epsilon} 
  \qandq
  \enscond{ t\in \TT }{ g(t)=-1 }\subset \ext^{-,\epsilon}.
\label{eq-g-proche}
\end{align}

%%%%%%%%%%%%%%%%
\paragraph{Variations of dual certificates.}
Let $p_\la$ be the solution of the noiseless problem~\eqref{eq-initial-dual} and $\tilde{p}_\la$ be the solution of the noisy dual problem $\Dd_\la(y_0+w)$ for $w\in L^2(\TT)$. Since the mapping $\frac{y_0}{\la} \mapsto p_\la$ is a projection onto a convex set (see~\eqref{eq-initial-dualbis}), it is non-expansive, \textit{i.e.} 
\begin{align}
	\norm{p_\la-\tilde{p}_{\la}}_2 \leq \frac{\norm{w}_2}{\la}.
  \label{eq-dual-nonexpansif}
\end{align}

As a consequence, if $\eta_\la = \Phi^* p_\la$ (resp. $\tilde{\eta}_\la = \Phi^* \tilde p_\la$) is the dual certificate of the noiseless (resp. noisy) problem, we have
\begin{align}\label{eq-dual-certif-control}
\norm{\eta_\la -\tilde{\eta}_\la}_\infty \leq M \frac{\norm{w}_2}{\la}
\end{align}
for some $M>0$ (in fact $M=\sqrt{\int_\TT |\varphi(t)|^2dt}= \norm{\Phi^*}_{\infty, 2}$).

From now on, we set $\alpha = \frac{r}{2M}$ and we impose $\frac{\norm{w}_2}{\la}\leq \alpha$.
Writing 
\begin{align*}
	\normi{\eta_0-\tilde{\eta}_\la}  & \leq \normi{\eta_0-\eta_\la} + 
		\normi{\eta_\la-\tilde{\eta}_\la},\\
		&\leq \normi{\eta_0-\eta_\la} + \frac{r}{2},
\end{align*}
we see using Proposition~\ref{prop-gamma-convergence} that for $\la$ small enough $\tilde{\eta}_\la$ satisfies~\eqref{eq-g-proche}.

%%%%%%%%%%%%%%%%
\paragraph{Structure of the reconstructed measure.}

By~\eqref{eq-g-proche} for $g=\tilde{\eta}_\la$ and using the extremality conditions we obtain that
$|\tilde{m}_\la|(K_\epsilon)=0$ and that $\tilde{m}_\la$ (resp. $-\tilde{m}_\la$) is non-negative in $\ext^{+,\epsilon}$ (resp. $\ext^{-,\epsilon}$).
Indeed, the extremality conditions impose that $\tilde{\eta}_\la= \sign \frac{d\tilde{m}_\la}{d|\tilde{m}_\la|}$, $\tilde{m}_\la$-almost everywhere,
hence the claimed result.
\end{proof}

Lemma~\ref{lem-robust-boites} does not make any assumption on the local structure
of $\exts m_0$, and does not provide any information on the local structure of $\tilde{m}_\la$ either (it might even not be discrete).
If we assume that $\eta_0'' (x)\neq 0$ for some $x\in \ext m_0$, then the reconstructed measure
has at most one spike in the neighborhood of $x$.

\begin{lem}
  Assume that there exists a solution to~\eqref{eq-constrained-dual} and that $\eta_0''(x)\neq 0$ for some $x\in \ext m_0$.
  Then for $\epsilon>0$ small enough, there exists $\alpha>0$, $\la_0>0$ such that
  for all $(\la, w)\in D_{\alpha, \la_0}$, the restriction of $\tilde{m}_\la$ to $(x-\epsilon,x+\epsilon)$ is
  \begin{itemize}
    \item either the null measure,
    \item or of the form $\tilde{a}_{\la, w}\delta_{\tilde{x}_{\la,w}}$ 
  where $\sign \tilde{a}_{\la, w}= \eta_0(x)$ and $\tilde{x}_{\la,w}\in (x-\epsilon,x+\epsilon)$.
\end{itemize}
If, in addition, $m_0$ is identifiable and $|m_0|((x-\epsilon,x+\epsilon))\neq 0$, only the second case may happen.
  \label{lem-nondegenerate}
\end{lem}
\begin{proof}
  The proof follows the same steps as those of Lemma~\ref{lem-robust-boites}.

%%%%%%%%%%%%%%%%
  \paragraph{Behavior of the minimal norm certificate.}
  First, observe that if $\eta_0''(x)\neq 0$  and $\eta_0(x)=1$ (resp. $-1$) for $x\in \ext m_0$, then $\eta_0''(x)<0$ (resp. $>0$).
  As a consequence, $x$ is an isolated point of $\ext m_0$.
  For $\epsilon>0$ small enough, $\ext m_0\cap (x-\epsilon,x+\epsilon)=\{x\}$
  and $|\eta_0''(t)|\geq \frac{|\eta_0''(x)|}{2}>0$ for all $t\in (x-\epsilon,x+\epsilon)$.

%%%%%%%%%%%%%%%%
\paragraph{Variations of dual certificates.}
From~\eqref{eq-dual-nonexpansif}, we infer that 
\begin{align}\label{eq-dual-certif-control2}
\norm{\eta_\la'' -\tilde{\eta}_\la''}_\infty \leq M \frac{\norm{w}_2}{\la}
\end{align}
with $M>0$ (here $M=\sqrt{\int_\TT |\varphi''(t)|^2 \d t}= \norm{(\Phi'')^*}_{\infty, 2}$).

We set $\alpha = \frac{r}{2M}$ with $r=\frac{|\eta_0''(x)|}{2}$ and we impose $\frac{\norm{w}_2}{\la}\leq \alpha$, so that
\begin{align*}
	\normi{\eta_0''-\tilde{\eta}_\la''}  & \leq \normi{\eta_0''-\eta_\la''} + 
		\normi{\eta_\la''-\tilde{\eta}_\la''},\\
		&\leq \normi{\eta_0''-\eta_\la''} + \frac{r}{2},
\end{align*}
thus $\normi{\eta_0''-\tilde{\eta}_\la''}<\frac{|\eta_0''(x)|}{2}$ for $\la$ small enough.

%%%%%%%%%%%%%%%%
\paragraph{Structure of the reconstructed measure.}
From the above inequality, we know that $\tilde{\eta}_\la$ is strictly concave (resp. strictly convex) in $(x-\epsilon,x+\epsilon)$.
As a result, there is at most one point $\tilde{x}_{\la, w}$  in $(x-\epsilon,x+\epsilon)$ such that $\tilde{\eta}_\la(\tilde{x}_{\la, w})=1$ (resp. $-1$).

If $m_0$ is identifiable, it remains to prove that there is indeed one spike in $(x-\epsilon,x+\epsilon)$.
This is obtained by relying on a result by Bredies and Pikkarainen~\cite{Bredies-space-measures} which is an application of~\cite[Th. 3.5]{Hofmann-measures}.
It guarantees that $\tilde{m}_{\la}$ converges to $m$ for the weak-* topology when $\la, \norm{w}_2\to 0$. 
We recall the result below (see Proposition~\ref{prop-hoffman-cvweak}) for the convenience of the reader.

By weak-* convergence of $\tilde{m}_\la$ to $m$ for $\la\to 0^+$ and $\norm{w}_2\to 0$, $\tilde{m}_{\la}((x-\epsilon,x+\epsilon))$ must converge to $m_0((x-\epsilon,x+\epsilon))$.
By the optimality conditions, we see that $|m_0((x-\epsilon,x+\epsilon))|=|m_0(\{x\})|$, so that $m_0(\{x\})\neq 0$ and $\sign m_0(\{x\})=\eta_0(x)$,
hence the result.
\end{proof}

In the proof of Lemma~\ref{lem-nondegenerate} we have relied on the following result.

\begin{prop}[\protect{\cite[Th. 3.5]{Hofmann-measures},\cite[Prop.~5]{Bredies-space-measures}}]
Let $m_0$ be an identifiable measure, if $\la \to 0$ and $\norm{w}\to 0$ with $\frac{\norm{w}_2^2}{\la}\to 0$, then
$\tilde{m}_\la$ converges to $m_0$ with respect to the weak-* topology.
\label{prop-hoffman-cvweak}
\end{prop}

%%%%%%%%%%%%%%%%%%%%%%%%%%%%%%%%%%%%%%%%%%%%%%
\subsection{Non Degenerate Source Condition}
\label{sec-source-cdt}

The notion of extended signed support has strong connections with the source condition introduced in~\cite{Burger_Osher04} to derive convergence rates for the Bregman distance.

\begin{defn}[Source Condition]
	A measure $m_0$ satisfies the source condition if there exists $p\in L^2(\TT)$ such that
	\eq{
		\Phi^* p \in \partial{\normTV{m_0}}.
	}
  \label{def-source-cdt}
\end{defn}

In a finite-dimensional framework, the source condition is simply equivalent to the optimality of 
$m_0$ for~\eqref{eq-constrained-pbm} given $y_0=\Phi m_0$. In the framework of Radon measures,
the source condition amounts to assuming that $m_0$ is a solution of~\eqref{eq-constrained-pbm} \textit{and} that there exists a solution to~\eqref{eq-constrained-dual}. 
In fact, the source condition simply means that the conditions of Proposition~\ref{prop-extend-solution} hold.

If one is interested in $m_0$ being the \textit{unique} solution of~\eqref{eq-constrained-pbm} for $y_0=\Phi m_0$ (in which case we say that $m_0$ is \textit{identifiable}), the source condition may be strengthened to give a sufficient condition.
 
\begin{prop}[\protect{\cite[Lemma 1.1]{deCastro-beurling}}]\label{prop-unique-sol}
Let $m_0 = m_{x_0,a_0}$ be a discrete measure. If $\Phi_{x_0}$ has full rank, and if
\begin{itemize}
	\item there exists $\eta \in \Im \Phi^*$ such that $\eta\in \partial{\normTV{m_0}}$,
	\item $\foralls s \notin \supp(m_0), \quad |\eta(s)| < 1$,
\end{itemize}
then $m_0$ is the unique solution of~\eqref{eq-constrained-pbm}.
\end{prop}

In this paper, in view of Lemma~\ref{lem-nondegenerate}, we strengthen a bit more the Source Condition so as to derive 
a global stability result concerning the support of the solutions of $\Pp(y_0+w)$
(see Theorem~\ref{thm-noise-robustness}).

\begin{defn}[Non Degenerate Source Condition]
Let $m_0=m_{x_0,a_0}$ be a discrete measure, and $\{x_{0,1},\ldots x_{0,N}\}=\supp m_0$. We say that $m_0$ satisfies the Non Degenerate Source Condition (NDSC) if
\begin{itemize}
\item there exists $\eta \in \Im \Phi^*$ such that $\eta \in \partial{\normTV{m_0}}$.

\item the minimal norm certificate $\eta_0$ satisfies
\begin{align*}
	&\foralls s \in \TT\setminus \{x_{0,1},\ldots x_{0,N}\}, \quad &|\eta_0(s)| < 1, \\
	&\foralls i\in \{1,\ldots N\}, \quad &\eta_0''(x_{0,i}) \neq 0.
\end{align*}
\end{itemize}
In that case, we say that $\eta_0$ is not degenerate.
\label{def-ndsc}
\end{defn}
The first assumption in the above definition is the standard Source Condition.
The last two assumptions impose conditions on the extended signed support, namely that $\ssupp m_0 = \exts (m_0)$ and 
$\eta_0''(t)\neq 0$ for all $t\in \supp m_0$.

When $\Phi$ is an ideal low-pass filter with cutoff frequency $f_c$, there are numerical evidences that measures having a large enough separation distance (proportional to $f_c$) satisfy the non degenerate source condition, see Section~\ref{sec-vanishing}.

%%%%%%%%%%%%%%%%%%%%%%%%%%%%%%%%%%%%%%%%%%%%%%%%%%%%%%%%%%%%%%%%%%%%%%%%%%%%%%%%%%%%%%%
\subsection{Main Result}

The following theorem, which is the main result of this paper, gives a global result on
the precise structure of the solution when the signal-to-noise ratio is large enough and $\la$ is small enough.

\begin{thm}[Noise robustness]
Let $m_0=m_{a_0,x_0}=\sum_{i=1}^N a_{0,i} \delta_{x_{0,i}}$ be a discrete measure.
Assume that $\Ga_{x_0}$ (defined in~\eqref{eq-gammax}) has full rank and that $m_0$ satisfies the Non Degenerate Source Condition.

Then there exists $\alpha>0, \la_0>0$, such that for $(\la,w) \in D_{\alpha,\la_0}$, the solution $\tilde{m}_\la$ of $\Pp_\la(y+w)$ is unique and is composed of exactly $N$ spikes. 

Moreover, up to a permutation of indices, we may write $\tilde{m}_\la=\sum_{i=1}^N \tilde{a}_{\la, i}\delta_{\tilde{x}_{\la,i}}$ with $\tilde{a}_{\la, i} \neq 0$ and $\sign(\tilde{a}_{\la, i})=\sign(a_{0, i})$ (for $1\leq i \leq N$), and writing $(\tilde{a}_0,\tilde{x}_0)=(a_0,x_0)$, the mapping 
\eq{
	(\lambda,w) \in D_{\alpha,\la_0} \mapsto (\tilde{a}_\la, \tilde{x}_\la ) \in \RR^N\times \TT^N, 
}
is $C^{k-1}$ whenever $\phi\in C^{k}(\TT)$ ($k\geq 2$).

In particular, for $\la =\frac{1}{\alpha}\norm{w}_2$, we have 
\begin{align}
 \forall i\in \{1,\ldots N\},\quad 
 |\tilde{x}_{\la,i}-x_{0,i}|=O(\norm{w}_2) \qandq |\tilde{a}_{\la,i}-a_{0,i}|=O(\norm{w}_2).
\end{align}
\label{thm-noise-robustness}
\end{thm}

\begin{proof}
  Applying Lemma~\ref{lem-nondegenerate} at each point $x_{0,i}$ for $1\leq i\leq N$ and Lemma~\ref{lem-robust-boites},
  we see that for $\epsilon>0$ small enough, there exists $\alpha>0$, $\la_0>0$ such that $\tilde{m}_\la$ has at most one spike in each interval $(x_{i,0}-\epsilon,x_{i,0}+\epsilon)$, and 
\eq{
  	|\tilde{m}_\la|\left(\TT\setminus\bigcup_{i=1}^N (x_{i,0}-\epsilon,x_{i,0}+\epsilon)\right)=0.
}

  In fact, since $\Ga_{x_0}$ has full rank, $\Phi_{\ext m_0}$ has full rank as well and $m_0$ is identifiable (by Proposition~\ref{prop-extend-solution}).
  Therefore, Lemma~\ref{lem-nondegenerate} ensures that there is indeed one spike in each interval, with sign equal to $\eta_0(x_{0,i})$.

%%%%%%%%%%%%%%%%
It remains to prove the uniqueness of the amplitudes and locations $(\tilde{a}_\la,\tilde{x}_\la)$ and their smoothness as function of $(\la, w)$.
To this end, we observe that they satisfy the following implicit equation
\eq{
	E_{s_0}(\tilde{a}_\la,\tilde{x}_\la,\la,w)=0
}
where $s_0=\sign (a_0)=(\eta_0(x_{i,0}))_{1\leq i\leq N}$, and  
\eq{
	E_{s_0}(a,x,\la,w) = 
	\begin{pmatrix}
		\Phi_x^* ( \Phi_x a - y_0 - w ) + \la s_0 \\
		{\Phi'_x}^* ( \Phi_x a - y_0 - w ) 
	\end{pmatrix}
	= 
	\Ga_x^*
	( \Phi_x a - y_0 - w ) 
	+ \la
	\begin{pmatrix}
		s_0\\
		0
	\end{pmatrix}.
}
Indeed, this implicit equation simply states that $\tilde{\eta}_\la(\tilde{x}_{\la,i})= \sign (a_{0,i})=\sign(\tilde{a}_{\la,i})$, and that $\tilde{\eta}_\la'(\tilde{x}_{\la,i})= 0$.

Since $((a,x),(\la,w)) \mapsto E_{s_0}(a,x,\la,w)$ is a $C^1$ function defined on $(\RR^{N}\times\TT^N) \times (\RR\times L^2(\TT^N))$, we may apply the implicit functions theorem.

The derivative of $E_{s_0}$ with respect to $x$ and $a$ reads
\begin{align*}
	\pd{E}{a}(a,x,\la,w)& = 
	\Ga_x^*
	\Phi_x \\
	\pd{E}{x}(a,x,\la,w) &= \begin{pmatrix}
    \diag({\Phi_x^{*}}' (\Phi_x a - y_0-w ))\\
    \diag({\Phi_x^{*}}'' (\Phi_x a - y_0-w ))
	\end{pmatrix}
	+ 
	\Ga_x^*
	{\Phi'_x} \diag(a).
\end{align*}
so that for $\la = 0$, $w=0$ and using $y_0 = \Phi_{x_0}a_0$, one obtains
\begin{align*}
	\pd{E_s}{(a,x)}(a_0,x_0,0,0) &= 
  \Ga_{x_0}^*
	\begin{pmatrix}
		\Phi_{x_0}, \: 
		\Phi'_{x_0} \diag(a_0)
	\end{pmatrix} \\
	& = 
	(\Ga_{x_0}^* \Ga_{x_0})
	\begin{pmatrix}
		\Id & 0 \\ 0 & \diag(a_0)
	\end{pmatrix}.
\end{align*}
Since we assume $\Ga_{x_0}$ has full rank, then $\pd{E_{s_0}}{(a,x)}(a_0,x_0,0,0)$ is invertible and the implicit functions theorem applies:
 there is a neighborhood $V\times W$ of $(a_0,x_0)\times \{(0,0)\}$ in $(\RR^{N}\times \TT^N)\times (\RR\times L^2(\TT))$
 and a function $f : W\rightarrow V$ such that
\begin{align*}
	& ((a,x),\lambda,w) \in V\times W \qandq E_{s_0}(a,x,\la,w)=0  \\
	\Longleftrightarrow \quad & (\la,w) \in W \qandq (a,x)=f(\la,w).
\end{align*}
Moreover, writing $(\hat a_{\lambda,w}, \hat x_{\lambda,w}) = f(\la,w)\in \RR^N\times \TT^N$, we have
\begin{itemize}
	\item $(\hat a_{0,0}, \hat x_{0,0}) = (a_0, x_0)$,
 	\item for any $(\la,w) \in W$, $\sign(\hat a_{\la,w}) = s_0$,
 	\item if $\phi\in C^k(\TT)$ (for $k\geq 2$), then $f\in C^{k-1}(W)$.
\end{itemize}
The constructed amplitudes and locations $(\hat a_{\lambda,w}, \hat x_{\lambda,w})$ coincide with those of the solutions of~$\Pp_\la(y_0+w)$ for all $(\la, w)\in W$ such that $\norm{w}_2\leq \alpha \la$.
 Possibly changing the value of $\la_0$ so that $D_{\alpha, \la_0}\subset W$, we obtain the desired result.
\end{proof}

\begin{rem}
Although this paper focuses on identifiable measures,  Theorem~\ref{thm-noise-robustness} 
describes the evolution of the solutions of $\Pp_\la(y_0+w)$ for any input measure $m_1$ such that there exists $m_0$
 which satisfies the non degenerate source condition and $y_0=\Phi m_1= \Phi m_0$. Instead of converging towards $m_1$, the solutions will converge towards $m_0$.
\end{rem}

%%%%%%%%%%%%%%%%%%%%%%%%%%%%%%%%%%%%%%%%%%%%%%%%%%%%%%%%%%%%%%%%%%%%%%%%%%%%%%%%%%%%%%%
\subsection{Extensions}
\label{subsec-extensions}

Theorem~\ref{thm-noise-robustness} extends in a straightforward manner to higher dimensions, i.e. when replacing $\TT$ by $\TT^d$ for $d \geq 1$.  In the NDSC introduced in Definition~\ref{def-ndsc}, one should replace, for $i=1,\ldots,N$, the constraint $\eta_0''(x_{0,i}) \neq 0$ by the constraint that the Hessian $D^2 \eta_0(x_{0,i}) \in \RR^{d \times d}$ is invertible. 

The proof also extends to non-stationary filtering operators, i.e. which can be written as
\eq{
	\foralls t \in \TT^d, \quad
	\Phi m(t) = \int_{\TT^d} \phi(x,t) \d m(x)
}
where $\phi \in C^2(\TT^d \times \TT^d)$.

%%%%%%%%%%%%%%%%%%%%%%%%%%%%%%%%%%%%%%%%%%%%%%%%%%%%%%%%%%%%%%%%%%%%%%%%%%%%%%%%%%%%%%%
\subsection{Application to the ideal Low-pass filter}

We first observe that the injectivity condition on $\Ga_x$ assumed in Theorem~\ref{thm-noise-robustness} always holds.
\begin{prop}[Injectivity of $\Ga_{x}$]
Let $x=(x_1,\ldots x_N)\in \TT^N$ with $x_i\neq x_j$ for $i\neq j$ and $N\leq f_c$.
Then $\Ga_x=(\Phi_x, \Phi_x')$ has full rank.
\label{prop-gamma-injective}
\end{prop}
The proof is given in Appendix~\ref{sec-proof1}.

\begin{figure}[htb]
\centering
\begin{tabular}{@{}c@{}c@{}}
\includegraphics[width=0.48\linewidth,clip,trim=0 -27px 0 0]{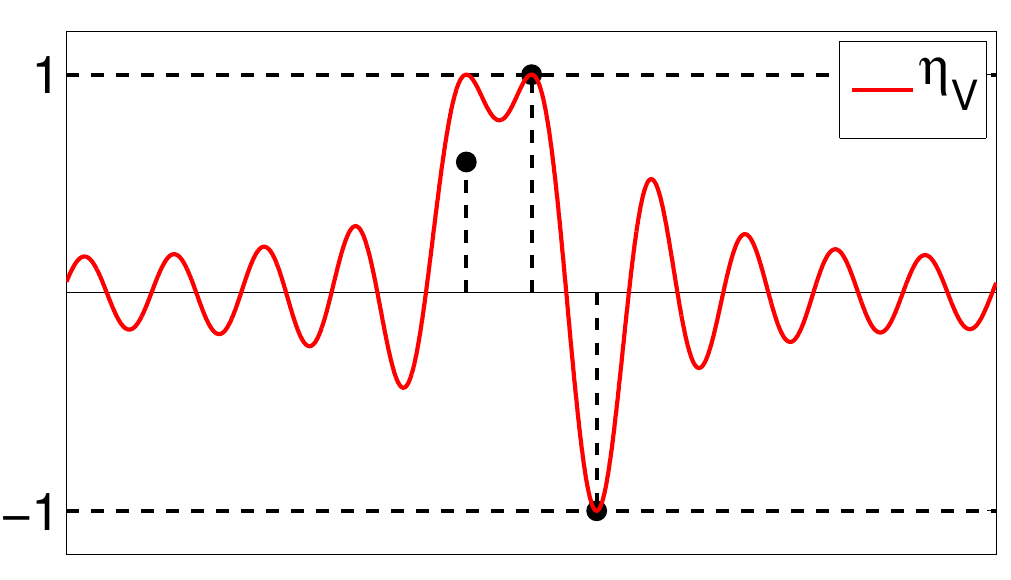}&
\includegraphics[width=0.49\linewidth]{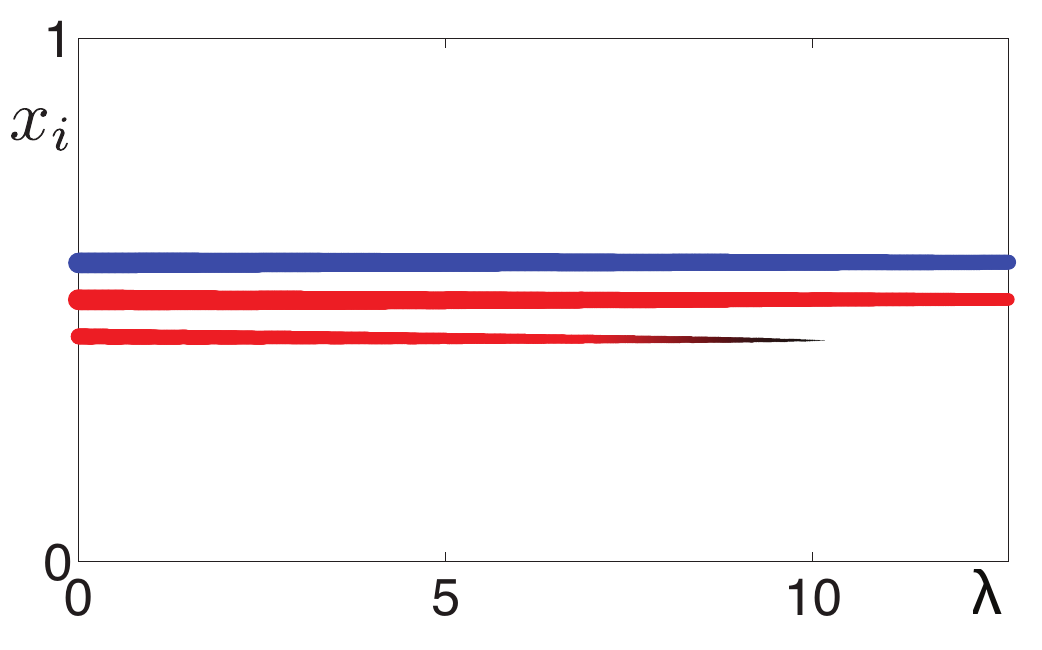}\\
(a) $m_0$ and $\eta_0$ & (b) $w=0$ \\
\includegraphics[width=0.49\linewidth]{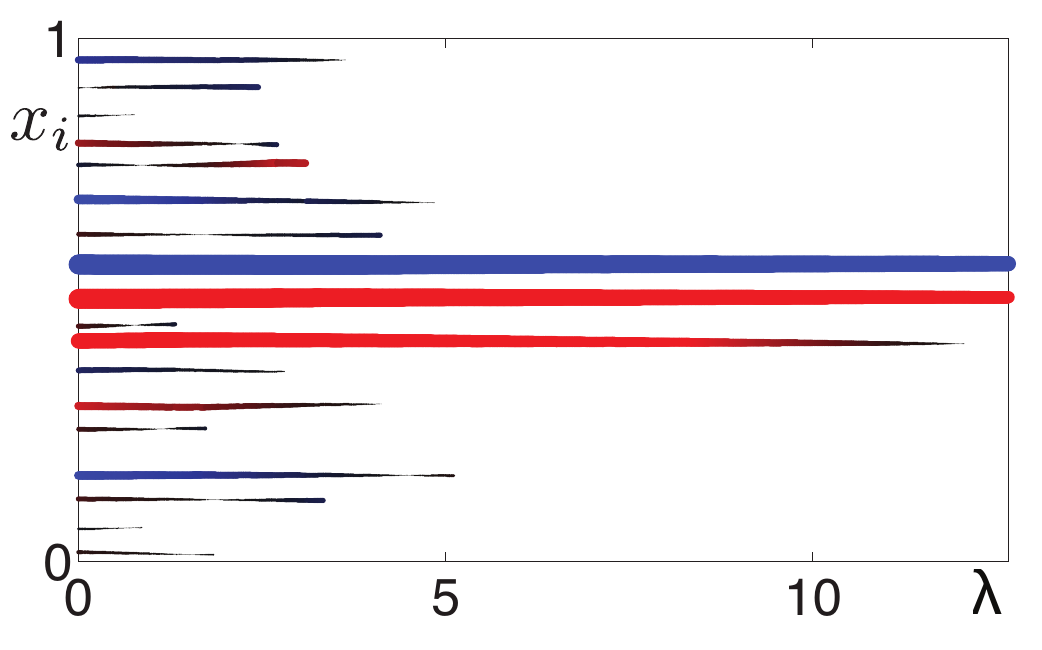}&
\includegraphics[width=0.49\linewidth]{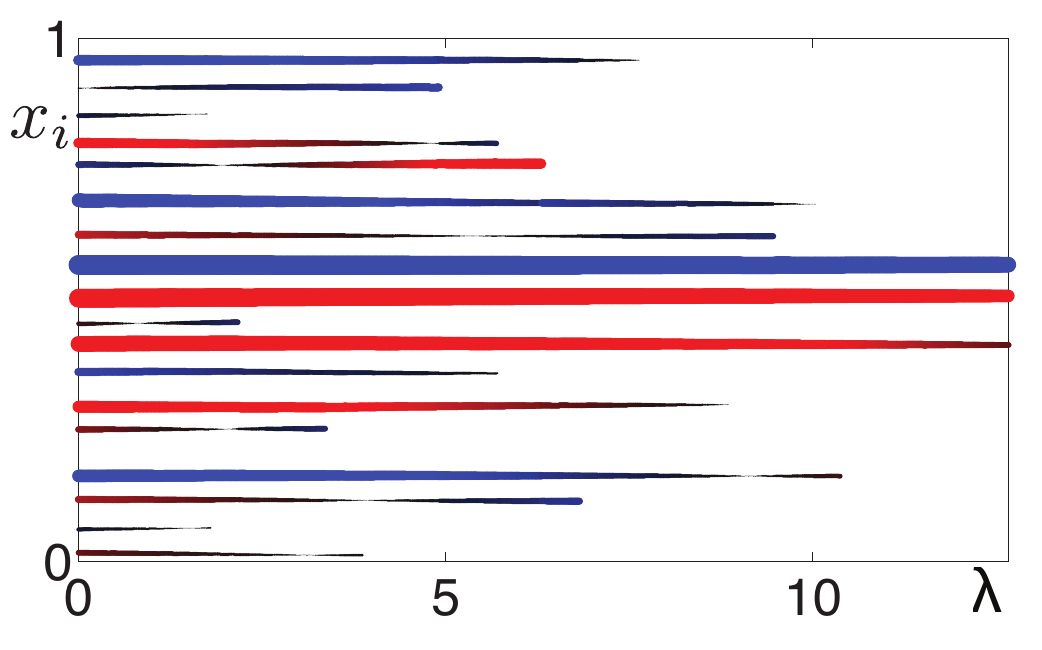}\\
(c) $\norm{w}=0.4 \norm{y}$ & (d) $\norm{w}=0.8 \norm{y}$
\end{tabular}
\caption{\label{fig-paths} 
(a) Input measure $m_0$, and corresponding minimal norm certificate. 
(b,c,d) Regularization paths $\la \mapsto \tilde{m}_\la$ that are solutions of $\Pp_\la(\Phi m_0 + w)$ for three different noise levels $\norm{w}$. 
Each ``strip'' represents the evolution of a spike as $\la$ varies. The color refers to the sign of the spike (blue for negative and red for positive)
and the (vertical) width is proportional to its amplitude.
The exact location is given by the middle of each band.
%The horizontal axis is $\la$, and the vertical axis indexes the Diracs' locations. Each vertical slice of these plots display graphically $m_\la$ with a dot for each Dirac location with a size and color (with blue for negative and red for positive) is proportional to the spike weights.
}
\end{figure} 

As to whether or not the Non Degenerate Source Condition holds for discrete measures, we will discuss this matter in Section~\ref{sec-vanishing} more in depth. For now, let us mention that we have observed empirically that this condition holds under the hypotheses of Theorem~$1.2$ in \cite{Candes-toward}, namely that $\Delta(m)>\frac{1.87}{f_c}$, but also with measures with far smaller values of $\Delta(m)$.

\begin{figure}[htb]
\centering
\begin{tabular}{@{}c@{}c@{}}
\includegraphics[width=0.48\linewidth,clip,trim=0 -27px 0 0]{paths/3diracsa-fc10-certif}&
\includegraphics[width=0.49\linewidth]{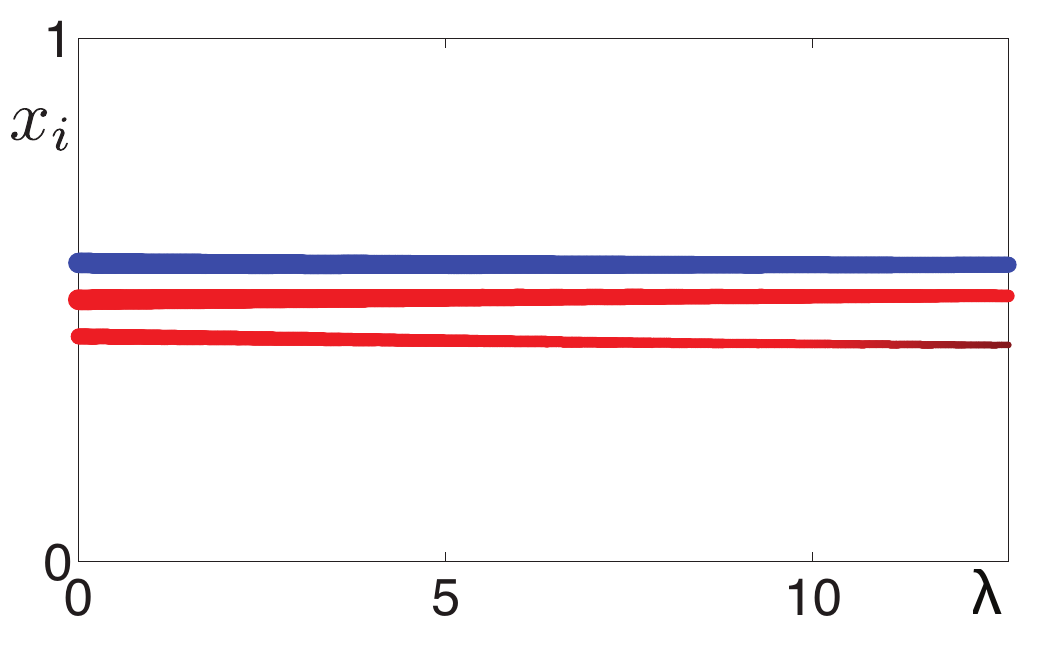}\\
(a) $m_0$ and $\eta_0$ & (b) $\norm{w_0}=0.07 \norm{y}$ \\
\includegraphics[width=0.49\linewidth]{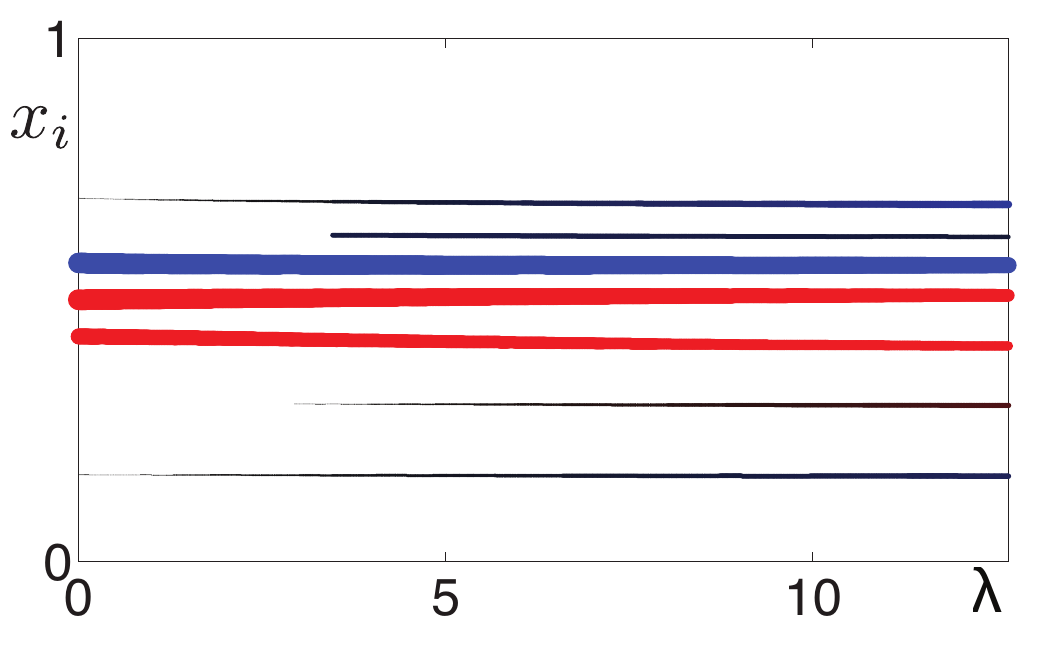}&
\includegraphics[width=0.49\linewidth]{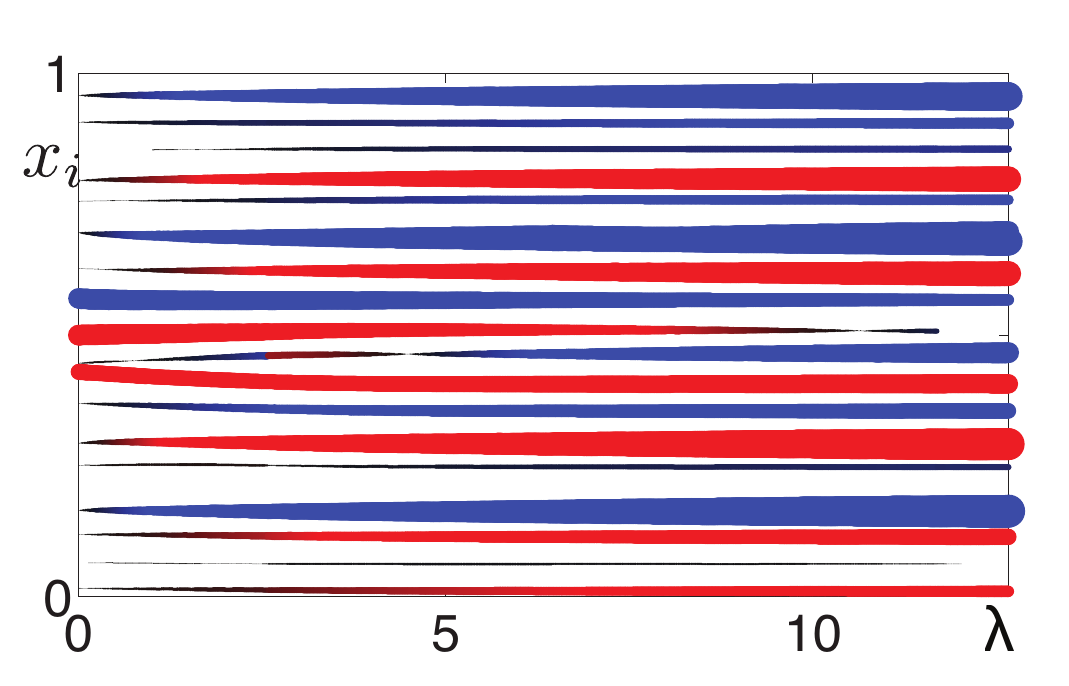}\\
(c) $\norm{w_0}=0.1 \norm{y}$ & (d) $\norm{w_0}=0.5 \norm{y}$
\end{tabular}
\caption{\label{fig-paths-scaling} 
Same plots as Figure~\ref{fig-paths} except that the solutions of $\Pp(\Phi m_0 + \la w_0)$ are displayed instead of those of $\Pp_\la(\Phi m_0 + w)$. 
}
\end{figure} 

Figure~\ref{fig-paths} shows the whole solution path $\la \mapsto \tilde{m}_\la$ of the solutions of $\Pp_\la(\Phi m_0 + w)$ when $f_c=10$ and the input measure is identifiable and has three spikes separated by $\Delta(m_0)=0.7/f_c$. Such a measure satisfies the Non-degenerate Source Condition as shown in plot (a). The plots (b,c,d) illustrate the conclusion of Theorem~\ref{thm-noise-robustness}. For values of $\la$ which are too small with respect to $\norm{w}$, the solution $\tilde{m}_\la$ is perturbed with spurious spikes, but as soon as $\la$ is large enough, $\tilde{m}_\la$ has a support that closely (but not exactly) matches the one of $m_0$. For large value of $\la$, spikes starts disappearing, and the support is not correctly estimated. Figure~\ref{fig-paths-scaling} shows the solutions of $\Pp_\la(\Phi m_0 + \la w_0)$, i.e. the noise $w=\la w_0$ is scaled by the regularization parameter $\la$. In accordance with Theorem~\ref{thm-noise-robustness}, this shows that for $\|w\|_2/\la = \|w_0\|_2 \leq 0.07$, the support of the spikes is precisely estimated.

% Figure~\ref{fig-robust} shows an example of the behavior predicted by Theorem~\ref{thm-noise-robustness}. Here, in order to illustrate Theorem~\ref{thm-noise-robustness}, for each value of $\la$, we have rescaled the noise so that $\norm{w}=\alpha \la$ (with $\alpha=0.7$). The frequency cutoff is $f_c=13$. First, the number of spikes may differ, with some spikes missing or appearing at strange places. But as $\la$ decreases, the number of spikes becomes the same as the original signal, and their positions and amplitudes of the estimated spikes converge to those of the true spikes.
 
% Several animated examples can be downloaded at \url{http://www.ceremade.dauphine.fr/~vduval/spikes.html}.

% !TEX root = ../DuvalPeyre-SparseSpikes.tex

\section{Vanishing Derivatives Pre-certificate}
\label{sec-vanishing}

We show in this section that, if the Non Degenerate Source Condition holds, the minimal norm certificate $\eta_0$ is characterized by its values on the support of $m_0$ and the fact that its derivative must vanish on the support of $m_0$. Thus, one may compute the minimal norm certificate simply by solving a linear system, without handling
    the cumbersome constraint $\normi{\eta_0} \leq 1$.

%%%%%%%%%%%%%%%%%%%%%%%%%%%%%%%%
\subsection{Dual Pre-certificates}
Loosely speaking, we call \textit{pre-certificate} any ``good candidate'' for a solution of \eqref{eq-extremal-constrained}.
Typically, a pre-certificate is built by solving a linear system (with possibly a condition on its norm).
The following pre-certificate appears naturally in our analysis.

\begin{defn}[Vanishing derivative pre-certificate]\label{defn-vanishing-der-certif}
The vanishing derivative pre-certificate associated with a measure $m_0 = m_{a_0,x_0}$ is $\etaV = \Phi^* \pV$ where
\begin{align}
	\pV =\uargmin{p \in L^2(\TT)} \norm{p}_2
		\quad \text{subj. to} \quad \foralls 1\leq i\leq N,\quad
	\choice{
		 (\Phi^* p)(x_{0,i})=\sign(a_{0,i}), \\
		 (\Phi^* p)'(x_{0,i})=0.
	}
\label{eq-vanishing-der-certif}
\end{align}
\end{defn}
It is clear that if the Source Condition (see Definition~\ref{def-source-cdt}) holds, then $\pV$ exists (since Problem~\eqref{eq-vanishing-der-certif} is feasible).
Observe that, in general, $\etaV$ is not a certificate for $m_0$ since it does not satisfy the constraint $\|\etaV\|_\infty\leq 1$.
The following proposition gathers several facts about the vanishing derivative pre-certificate which show that it is indeed a good candidate for the minimal norm certificate.

\begin{prop}
Let $m_0=m_{a_0,x_0}=\sum_{i=1}^N a_{0,i} \delta_{x_{0,i}}$ be a discrete measure.
The following assertions hold.
\begin{itemize}
  \item Problem~\eqref{eq-vanishing-der-certif} is feasible and $\|\etaV\|_\infty\leq 1$ if and only if the Source Condition holds and $\etaV=\eta_0$.
  \item If Problem~\eqref{eq-vanishing-der-certif} is feasible and $\Ga_{x_0}$ has full rank, i.e. $\Ga_{x_0}^* \Ga_{x_0} \in \RR^{2N \times 2N}$ is invertible, then 
    	\eq{
		\etaV = \Phi^* \Ga_{x_0}^{+,*} 
		\begin{pmatrix}
			\sign(a_0) \\
			0
		\end{pmatrix}
		\qwhereq
		\Ga_{x_0}^{+,*} = \Ga_{x_0}(\Ga_{x_0}^* \Ga_{x_0})^{-1}.
	}
\item If $\Ga_{x_0}$ has full rank, then $m_0$ satisfies the Non Degenerate Source Condition if and only if Problem~\eqref{eq-vanishing-der-certif} is feasible and 
\begin{align*}
	&\foralls s \in \TT\setminus \{x_{0,1},\ldots x_{0,N}\}, \quad &|\etaV(s)| < 1, \\
	&\foralls i\in \{1,\ldots N\}, \quad &\etaV''(x_{0,i}) \neq 0.
\end{align*}
\end{itemize}
\label{prop-vanish-tools}
\end{prop}

The third assertion of Proposition~\ref{prop-vanish-tools} states that it is equivalent to check the Non Degenerate Source Condition on $\eta_0$ (Definition~\ref{def-ndsc}) or to check the same conditions on $\etaV$. In case those conditions hold, one even has $\etaV=\eta_0$ (first assertion). The main point of this equivalence is that the second assertion yields a practical expression to compute $\etaV$ which may be used in numerical experiments (see Section~\ref{sec-vanishing-application}).

\begin{proof} 
  For the first assertion, we observe that if Problem~\eqref{eq-vanishing-der-certif} is feasible (and thus $\pV$ exists) and $\|\etaV\|_\infty\leq 1$, then $\etaV\in \partial{\normTV{m_0}}$ and the Source Condition holds. Hence, $\|\pV\|_2\geq \|p_0\|_2$. On the other hand the minimal norm certificate $\eta_0$ must satisfy all the constraints of \eqref{eq-vanishing-der-certif},  thus the minimality of the norms of both $\etaV$ and $\eta_0$ implies that $\etaV=\eta_0$. The converse implication is obvious.

For the second assertion, Problem~\eqref{eq-vanishing-der-certif} can be written as
\eq{
	\etaV =  \uargmin{ \eta = \Phi^* p } \norm{p}_2.
	\quad \text{subj. to} \quad
	\choice{
		\Phi_{x_0}^*p = \sign(a_0), \\
		\Phi_{x_0}'^*p = 0,
	}
}
which is a quadratic optimization problem in a Hilbert space with a finite number of affine equality constraints.
Moreover, the assumption that $\Ga_{x_0}$ has full rank implies that the constraints are qualified. Hence it can be solved by introducing Lagrange multipliers $u$ and $v$ for the constraints. One should therefore solve the following linear system to obtain the value of $p=\pV$
	\eq{
		\begin{pmatrix}
			\Id 		& \Phi_{x_0} 	& \Phi'_{x_0} \\
			\Phi_{x_0}^*	& 0 	&0 \\
			{\Phi'_{x_0}}^*	& 0 	&0 \\
		\end{pmatrix}
		\begin{pmatrix}
			p \\ u \\ v
		\end{pmatrix}
		=
		\begin{pmatrix}
			0 \\ s \\ 0
		\end{pmatrix}.
	}
	Solving for $(u, v)$ in these equations gives the result.

  For the third assertion, if the Non Degenerate Source condition holds, we apply Theorem~\ref{thm-noise-robustness} which yields a $C^1$ path $\la \mapsto (\tilde{a}_\la,\tilde{x}_\la )$ of solutions of $\Pp_\la(y_0)$ (we consider the case $w=0$). Then from Proposition~\ref{prop-vanish-certif} below, we obtain that $\etaV$ is a valid certificate and $\etaV=\eta_0$, hence $\etaV$ is non-degenerate.
The converse implication is a straightforward consequence of the first assertion.
\end{proof}

%%%%%%%%%%%%%%%%%%%%%%%%%%%%%%%%%%%%%%%%%%%%%%%%%%%%%%%%%%%%%%%%%%%%%%%%%%%%
\subsection{Necessary condition for support recovery}

There is a priori no reason for the vanishing derivative pre-certificate $\etaV$ to satisfy
$\normi{\etaV}\leq 1$. Here, we prove that that is in fact a necessary condition for (noiseless) exact support recovery to hold on some interval $[0,\la_0)$ with $\la_0>0$, \textit{i.e.} the solutions of $\Pp_\la(y_0)$ having exactly $N$ spikes which converge smoothly towards those of the original measure.
 
\begin{prop}
Let $m_0=m_{a_0,x_0}=\sum_{i=1}^N a_{0,i} \delta_{x_{0,i}}$ be a discrete measure such that $\Ga_{x_0}$ has full rank. 
Assume that there exists $\la_0>0$ and a $C^1$ path $[0,\la_0)\rightarrow \RR^N\times \TT^N$, $\la \mapsto (a_\la, x_\la)$ such that for all $\la\in [0,\la_0)$ the measure $m_\la=m_{a_\la,x_\la}$ is a solution to $\Pp_\la(y_0)$ (the noiseless problem).
  
  Then $\etaV$ exists, $\|\etaV\|_\infty\leq 1$ and $\etaV=\eta_0$.
\label{prop-vanish-certif}
\end{prop}

\begin{proof}
  Let $p_\la =\frac{1}{\la} (y_0-\Phi m_\la)=\frac{1}{\la} (\Phi_{x_0}a_0 -\Phi_{x_\la} a_\la)$ be the certificate defined by the optimality conditions~\eqref{eq-extremal-cdt}. We show that $\Phi^* p_\la$ converges towards $ \Phi^* \Ga_{x_0}^{+,*} 
		\begin{pmatrix}
			\sign(a_0) \\
			0
		\end{pmatrix}=\etaV$
 (and that the latter exists).

Writing
	\eq{
    a_\la' = \frac{\d a_\la}{\d\la} \in \RR^N
		\qandq
		x_\la' = \frac{\d x_\la}{\d\la} \in \RR^N,
	}
 we observe that for any $i\in \{1,\ldots N\}$ and any $x\in \TT$, 
\begin{align*}
	&\frac{a_{\la,i}\varphi (x_{\la, i}-x) - a_{0,i}\varphi (x_{0, i}-x)}{\la}- \left[ a_{0,i}\varphi' ( x_{0, i}-x) x_{0,i}' + a_{0,i}'\varphi ( x_{0, i}-x) \right]\\
 &\quad = \int_{0}^1 \left[ a_{\la t,i}\varphi' ( x_{\la t, i}-x) x_{\la t,i}' + a_{\la t,i}'\varphi ( x_{\la t, i}-x) \right]\\
 & \qquad \qquad - \left[ a_{0,i}\varphi' ( x_{0, i}-x) x_{0,i}' + a_{0,i}'\varphi ( x_{0, i}-x) \right]  \d t,
\end{align*}
and the latter integral converges (uniformly in $x$) to zero when $\la \to 0^+$ by uniform continuity of its integrand (since $ a$, $ x$ and $\varphi$ are $C^1$).
As a consequence, we obtain that $\frac{y_0-\Phi_{ x_\la}{a}_\la}{\la}$ converges uniformly to $-\Ga_{x_0}
		\begin{pmatrix}
			\Id & 0 \\ 0 & \diag(a_0)
		\end{pmatrix}
		\begin{pmatrix}
			 a_0' \\  x_0'
		\end{pmatrix}$.

    On the other hand, we observe that for $\la$ small enough, $\sign (a_\la)=\sign (a_0)$, and using the notations of the proof of Theorem~\ref{thm-noise-robustness}, the implicit equation $E_{s_0}(a_\la,x_\la,\la,0)=0$ holds. Differentiating that equation at $\la=0$ we obtain:
\begin{align*}	
  \pa{ \pd{E_{s_0}}{(a,x)}(a_0,x_0,0,0) }
			\begin{pmatrix}
			 a_0' \\  x_0'
		\end{pmatrix}
	+\pd{E_{s_0}}{\la}(a_0,x_0,0,0)&=0, 
\end{align*}
or equivalently
\begin{align*}	
	(\Ga_{x_0}^* \Ga_{x_0})\begin{pmatrix}
		\Id & 0 \\ 0 & \diag(a_0)
	\end{pmatrix}
\begin{pmatrix}
			 a_0' \\  x_0'
		\end{pmatrix}
&= - 	\begin{pmatrix}
		s_0\\
		0
	\end{pmatrix}.
\end{align*}

As a consequence, Problem~\eqref{eq-vanishing-der-certif} is feasible and we see that $\frac{y_0-\Phi_{ x_\la}{a}_\la}{\la}$ converges uniformly (and thus in the $L^2$ strong topology) to $\Ga_{x_0}(\Ga_{x_0}^* \Ga_{x_0})^{-1}\begin{pmatrix}
			\sign(a_0) \\
			0
		\end{pmatrix}$ and $\Phi^*\left(\frac{y_0-\Phi_{ x_\la}}{\la}\right)$ converges uniformly to $\Phi^* \Ga_{x_0}^{+,*} 
		\begin{pmatrix}
			\sign(a_0) \\
			0
    \end{pmatrix}$ (which is precisely $\etaV$ from the second assertion of Proposition~\ref{prop-vanish-tools}).

   Since $\|\Phi^* \left(\frac{y_0-\Phi_{ x_\la}a_\la}{\la}\right)\|_\infty = \|\Phi^*p_\la\|_\infty\leq 1$ for all $\la >0$, we obtain that $\|\etaV\|_\infty\leq 1$, hence the claimed result.
\end{proof}

%%%%%%%%%%%%%%%%%%%%%%%%%%%%%%%%%%%%%%%%%%%%%%%%%%%%%%%%%%%%%%%%%%%%%%%%%%%%
\subsection{Application to the Ideal Low-pass Filter}
\label{sec-vanishing-application}

In order to prove their identifiability result for measures, the authors of~\cite{Candes-toward} also introduce a ``good candidate'' for a dual certificate associated with $m=m_{a,x}$ for $a \in \CC^N$ and $x \in \RR^N$. 
For $K$ being the square of the Fejer kernel, they build a trigonometric polynomial
\eq{
	\etaCF(t)=\sum_{i=1}^N \left(\alpha_{i}K(t-x_i) + \beta_{i}K'(t-x_i) \right) \mbox{ with } K(t)=\left(\frac{\sin \left(\left(\frac{f_c}{2}+1 \right)\pi t\right)}{\left(\frac{f_c}{2}+1\right)\sin \pi t}\right)^4
}
and compute $(\al_i,\be_i)_{i=1}^N$ by imposing that $\etaCF(x_i)=\sign (a_i)$ and $(\etaCF)' (x_i)=0$.

They show that the constructed pre-certificate is indeed a certificate,  i.e. that $\normi{\etaCF} \leq 1$, provided that the support is separated enough (i.e. when $\Delta(m)\geq C/f_c$). This result is important since it proves that measures that have sufficiently separated spikes are identifiable. Furthermore, using the fact that $\etaCF$ is not degenerate (i.e. $(\etaCF)''(x_i) \neq 0$ for all $i=1,\ldots,N$), the same authors derive  an $L^2$ robustness to noise result in~\cite{Candes-superresol-noisy}, and Fernandez-Granda and Azais et al.  use the constructed certificate to analyze finely the local averages of the spikes in~\cite{Fernandez-Granda-support,Azais-inaccurate}.

From a numerical perspective, we have investigated how this pre-certificate compares with the vanishing derivative pre-certificate that appears naturally in our analysis, by generating real-valued measures for different separation distances and observing  when each pre-certificate $\eta$ satisfies $\normi{\eta} \leq 1$.

%\begin{figure}[ht]
%\centering
%	\begin{tabular}{@{}c@{\hspace{1mm}}c@{}}
%	   \includegraphics[width=0.46\linewidth] {precertif-2p5.pdf} &
%   		\includegraphics[width=0.46\linewidth] {precertif-1p26.pdf} \\
%		$\Delta(m)\approx 2.50/f_c$ & $\Delta(m)\approx 1.26/f_c$ \\[3mm]
%	   \includegraphics[width=0.46\linewidth] {precertif-0p69.pdf} &
%	   \includegraphics[width=0.46\linewidth] {precertif-0p44.pdf}\\
%		$\Delta(m)\approx 0.69/f_c$ & $\Delta(m)\approx 0.44/f_c$
%	\end{tabular}
%\caption{\label{fig-precertif} % 
%The blue curve is the square Fejer kernel pre-certificate $\hat \eta_0$ (introduced in~\cite{Candes-toward}) and the green curve shows the vanishing derivative pre-certificate $\bar\eta_0$, defined in~\eqref{eq-vanishing-der-certif}, for several measures $m$ (whose Diracs' locations and elevation are display in dashed red) with different separation distances $\De(m)$. Eventually when $\Delta(m)$ is small enough, both pre-certificates break (i.e. are not anymore certificates), but the square Fejer always breaks before the vanishing derivative pre-certificate.\vspace{2mm}}
%\end{figure}

\begin{figure}[ht]
\centering
% \begin{tabular}{@{}c@{}c@{}}
\subfloat[$\Delta(m_0)=0.8/f_c$]{\includegraphics[width=0.49\linewidth]{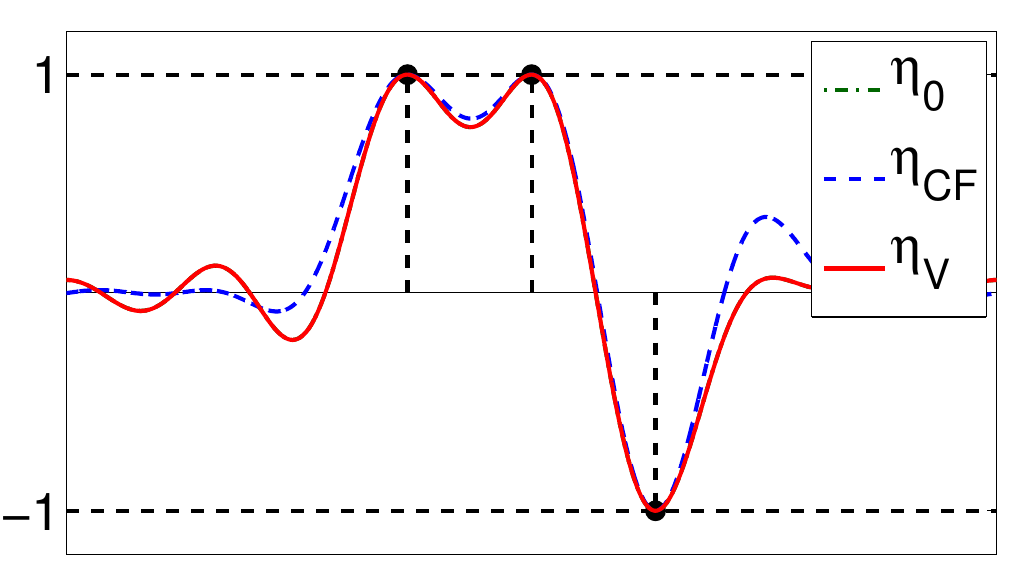}}
\subfloat[$\Delta(m_0)=0.7/f_c$]{\includegraphics[width=0.49\linewidth]{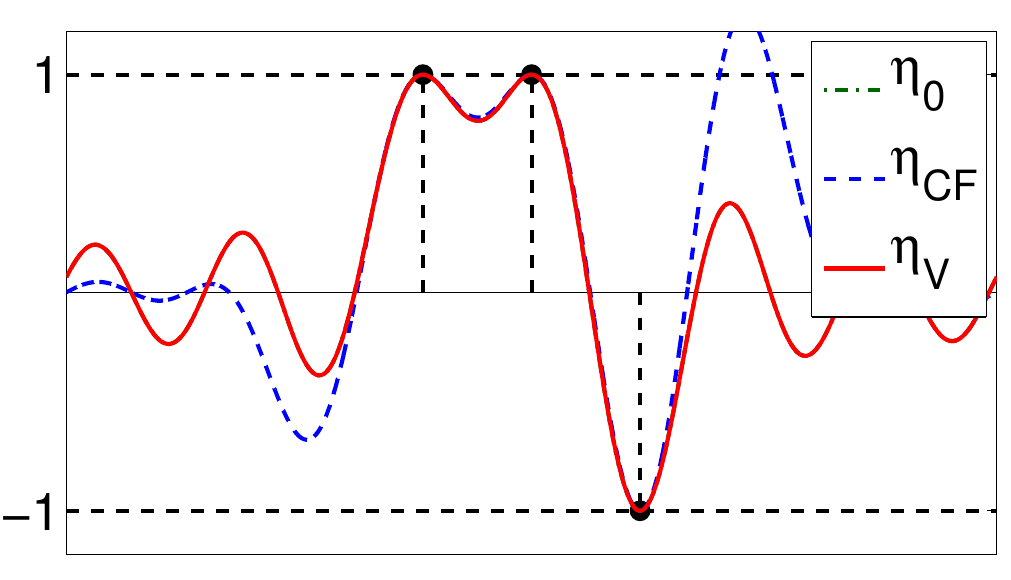}}\\
\subfloat[$\Delta(m_0)=0.6/f_c$]{\includegraphics[width=0.49\linewidth]{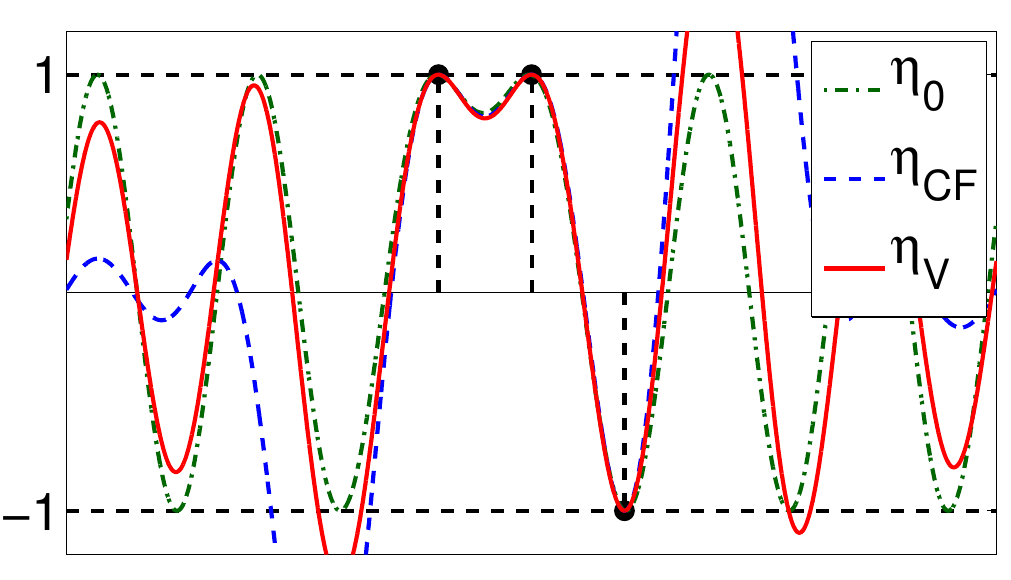}}
\subfloat[$\Delta(m_0)=0.5/f_c$]{\includegraphics[width=0.49\linewidth]{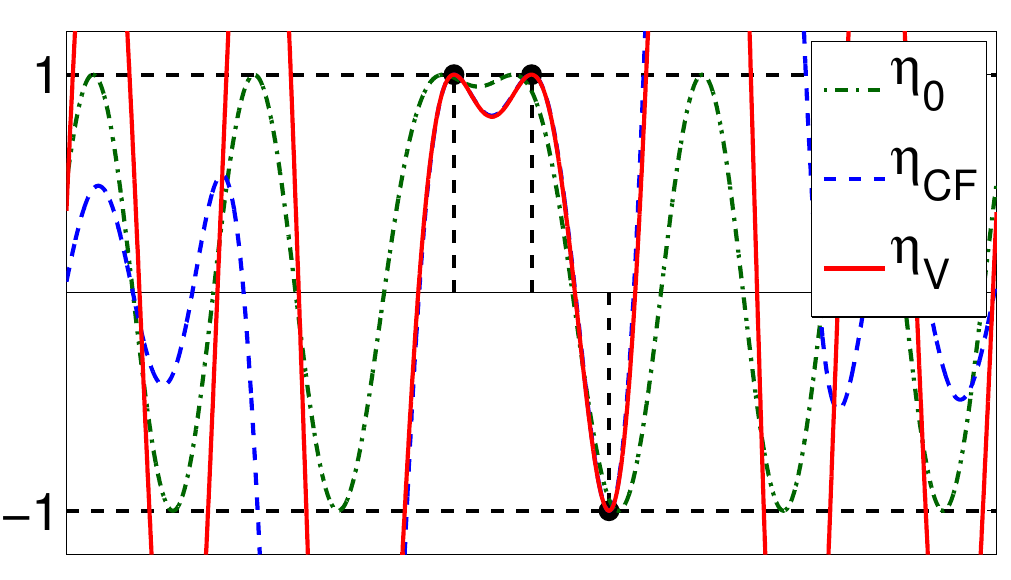}}
\caption{\label{fig-certif} % 
Pre-certificates for three equally spaced masses. The blue curves with dots is the Fejer
 pre-certificate $\etaCF$, while red continuous line is the vanishing derivative $\etaV$. The 
 black dashed line is the minimal norm precertificate $\eta_0$.\vspace{2mm}}
\end{figure}

As predicted by the result of~\cite{Candes-toward}, we observe numerically that the pre-certificate
 $\etaCF$ is a certificate (i.e. $\normi{\etaCF} \leq 1$) for any measure with $\De(m_0) \geq 1.87/f_c$.
 We also observe that this continues to hold up to $\De(m_0) \geq 1/f_c$. Yet, below $1/f_c$,
 it may happen that some measures are still identifiable (as asserted using the
 vanishing derivative pre-certificate $\etaV$)  but $\etaCF$ stops being a certificate, i.e.
 $\normi{\etaCF} > 1$. A typical example is shown in Figure~\ref{fig-certif}, where, for $f_c=6$
 we have used three equally spaced masses as an input, their separation distance being
 $\Delta(m_0)\in \{\frac{0.8}{f_c},\frac{0.7}{f_c},\frac{0.6}{f_c},\frac{0.5}{f_c}\}$. 
 Here, we have computed an approximation of the minimal norm certificate $\eta_0$ by solving
  \eqref{eq-initial-dual} with very small $\la$.
 
 For $\Delta(m_0)=\frac{0.8}{f_c}$, both $\etaV$ and $\etaCF$ are certificates, so that the
vanishing derivatives pre-certificate $\etaV$ is equal to the minimal norm certificate $\eta_0$.
For $\Delta(m_0)=\frac{0.7}{f_c}$, $\etaCF$ violates the constraint $\normi{\etaCF}\leq 1$ but
 the vanishing derivative pre-certificates is still a certificate (even showing that the measure
is identifiable). For $\Delta(m_0)=\frac{0.6}{f_c}$ and $\frac{0.5}{f_c}$, neither $\etaV$ nor $\etaCF$
satisfy the constraint, hence $\etaV \neq \eta_0$. Yet, $\eta_0$ ensures that $m_0$ is a solution to \eqref{eq-constrained-pbm}.

 % Yet, below $1/f_c$,
% we observe numerically that some measures are still identifiable (as asserted using the
% vanishing derivative pre-certificate $\etaV$)  but $\etaCF$ stops being a certificate, i.e.
% $\normi{\etaCF} > 1$. An illustration is given in Figure~\ref{fig-precertif}, where the
%  chosen parameters are $f_c=26$ and $N=7$. For the cases $\Delta(m)=2.50/f_c$ and
%   $\Delta(m)=1.26/f_c$, both pre-certificates $\etaV$ and $\etaCF$ are certificates,
% showing that the generated measure is identifiable. Notice how the vanishing derivative
% certificate $\etaV$ oscillates much more than the square Fejer certificate $\etaCF$ . For
% $\Delta(m)=0.69/f_c$, the square Fejer pre-certificate breaks the constraint
%  ($\normi{\etaCF} \approx 2.39$) whereas the vanishing derivative certificate still satisfies
% $\normi{\etaV} \leq 1$. Eventually, for $\Delta(m)=0.44/f_c$, both pre-certificates violate
% the constraint, with $\normi{\etaCF} \approx 3.39$ and $\normi{\etaV}=1.17$.
% 

\begin{figure}[ht]
\centering
\subfloat[$m$ and certificates]{\includegraphics[width=0.49\linewidth]{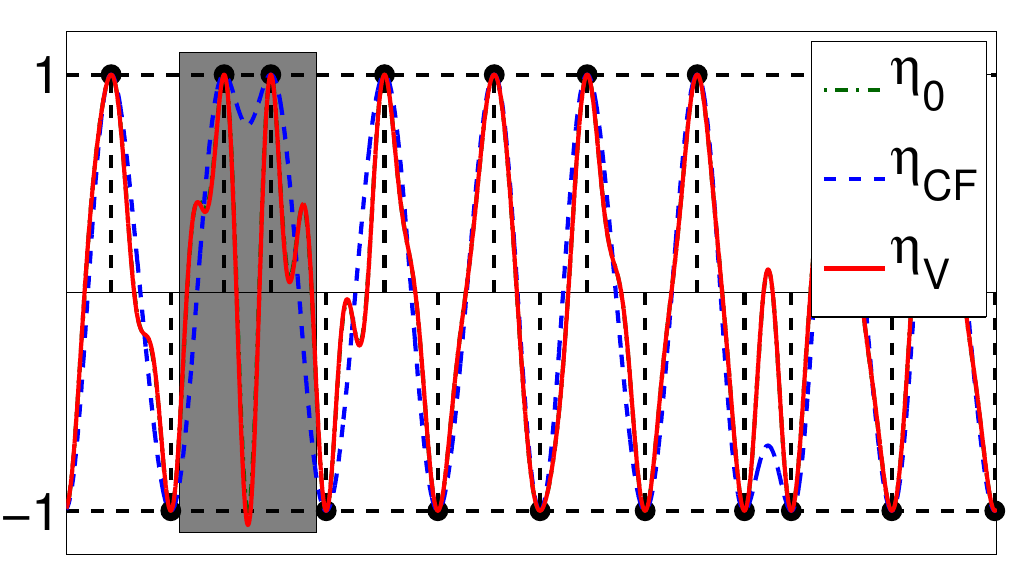}}
\subfloat[Zoom on the gray area]{\includegraphics[width=0.49\linewidth]{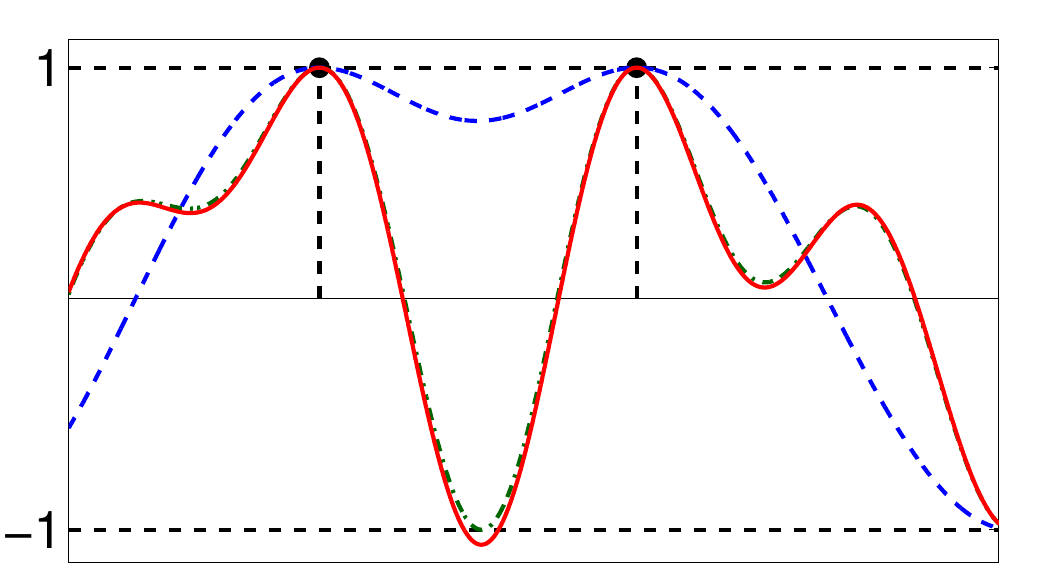}}
\caption{\label{fig-certif-pathologic} % 
	Example of measure for which $\etaV \neq \eta_0$. }
\end{figure}

From the experiments we have carried out, we have observed that the vanishing derivative pre-certificate $\etaV$
 behaves in general at least as well as the square Fejer $\etaCF$. The only exceptions we have 
noticed is for a large number of peaks (when $N$ is close to $f_c$), with $\Delta(m_0)\leq \frac{1.5}{f_c}$.
This is illustrated in Figure~\ref{fig-certif-pathologic} which shows a measure $m_0$ for which $\etaCF$ is a non-degenerate certificate (which shows that it is identifiable), but for which $\eta_0 \neq \etaV$ since $\normi{\etaV}>1$ (thus $\etaV$ is not a certificate). Typically, we have in this case $\ssupp m_0 \subsetneq \exts(m_0)$. Such a measure is identifiable but there is no support recovery for $\la >0$ (in the sense of Proposition~\ref{prop-vanish-certif}), hence its support is not stable.
 
Such pathological cases are relatively rare. An intuitive explanation for this is the fact that having $\eta_0(x)=\pm 1$
for $x\in \TT\setminus \supp (m_0)$ or  $\eta_0''(x)=0$ for some $x\in \supp (m_0)$ tend to impose a large $L^2$ norm, thus contradicting the minimality of $p_0$ 
(recall that when $\phi$ is an ideal low pass filter $\norm{\eta}_2=\norm{p}_2$).

% !TEX root = ../DuvalPeyre-SparseSpikes.tex

%%%%%%%%%%%%%%%%%%%
\section{Discrete Sparse Spikes Deconvolution}
\label{sec-discrete}

%%%%%%%%%%%%%%%%%%%%%%%%%%%%%%%%%%%%%%%%%%%%%%%%%%%%%%%%%
\subsection{Finite Dimensional $\ell^1$ Regularization}

A popular way to compute approximate solutions to~\eqref{eq-initial-pb} with fast algorithms is to solve this problem on a finite discrete grid $\Gg\subset \TT$.
Denoting by $P$ the cardinal of the grid $\Gg$, and by $g\in \TT^P$ the finite
 sequence of elements of $\Gg$, the idea is to solve $\Pp_\la(y_0)$ (or~\ref{eq-constrained-pbm})
with the additional constraint that $m=\sum_{i=1}^P a_i \delta_{g_i}$ for some $a\in \RR^P$.

This is nothing but the so-called basis pursuit denoising problem~\cite{chen1999atomi}, also known as the Lasso~\cite{tibshirani1996regre} in statistics.
Indeed, defining the linear operator $\Psi$ through 
\eq{
	\Psi a = \Phi m= \sum_{i=1}^P (\Phi \delta_{g_i})a_i, 
}
the problem amounts to:
\eql{\label{eq-lasso-discr-finite}\tag{$\tilde{\Pp}_\la^\Gg(y_0)$}
	\umin{ a \in \RR^P } \frac{1}{2}\norm{y_0 - \Psi a}^2 + \la \norm{a}_1
	\qwhereq
	\norm{a}_1 = \sum_{i=1}^P |a_i|,
}
where $\Psi:\RR^P \rightarrow L^2(\TT)$ is a linear operator ($L^2(\TT)$ may as well be replaced with $\RR^Q$ or any Hilbert space),
 and $a_i$ denotes the mass at each point $i$ of the grid.
In the noiseless case, the exact reconstruction problem reads:
\eql{\label{eq-lasso-discr-noiseless-finite}\tag{$\tilde{\Pp}_0^\Gg(y_0)$}
 	\umin{ \Psi a = y_0 } \norm{a}_1.
}
The aim of the present section is to study the asymptotic of Problems~\eqref{eq-lasso-discr-finite} and~\eqref{eq-lasso-discr-noiseless-finite} as the stepsize of the grid $\Gg$ vanishes. To this end, we keep the framework of measures and we reformulate the constraint that $\supp(m) \subset \Gg$, i.e. that $m$ can be written as $m=m_{a,x}$, where $x = (x_1,\ldots,x_N)\in \Gg^N$. Recall that the notation $m_{a,x}$ hints that $a_i\neq 0$ for all $i$ and that the $x_i$'s are all distinct, so that in general $N\leq P$. We thus adopt the following penalization term
\begin{align}
	\normTVG{m}= \sup \enscond{ \int \psi \d m }{  
		\psi\in C(\TT), \forall t\in \Gg \ |\psi(t)| \leq 1 },
\end{align}
so that $\normTVG{m}=+\infty$ when $\supp(m) \not\subset \Gg$, and $\sum_{i=1}^N |a_i|$ otherwise.

Problems~\eqref{eq-lasso-discr-finite} and~\eqref{eq-lasso-discr-noiseless-finite} are then respectively  equivalent to:
\eql{\label{eq-lasso-discr}\tag{$\Pp_\la^\Gg(y_0)$}
	\umin{m \in \Mm(\TT)} \frac{1}{2} \norm{\Phi(m) - y_0}^2 + \la \normTVG{m}
}
and 
\eql{\label{eq-lasso-discr-noiseless}\tag{$\Pp_0^\Gg(y_0)$}
	\umin{ \Phi m = y_0 } \normTVG{m},
}

% Gab: removed "may be derived directly in a finite dimensional framework, and more importantly that they"
 
Let us stress the fact that the results of Sections~\ref{subsec-certif-discr} and~\ref{sec-discr-robustness} hold for \textit{any finite dimensional matrix} $\Psi \in \RR^{P \times Q}$ or linear operator $\Psi: \RR^P\rightarrow L^2(\TT)$: the columns of $\Psi$ need not be the samples of a convolution operator.

%%%%%%%%%%%%%%%%%%%%%%%%%%%%%%%%%%%%%%%%%%%%%%%%%%%%%%%%%
\subsection{Certificates over a Discrete Grid}
\label{subsec-certif-discr}

%\todo{The following notations needed to be introduced somewhere. Check where you want to put it.} 
% Je l'ai mis dans la partie notation
As in Section~\ref{sec-preliminaries}, we may compute the subdifferential of the $\ell^1$ norm. For $m=m_{a,x}=\sum_{i=1}^N a_i\delta_{x_i}$ with support in $\Gg$:
\begin{align}
 	\partial \normTVG{m} = \enscond{\eta\in C(\TT)}{ \norm{\eta}_{\infty,\Gg} \leq 1, 
		\foralls i=1,\ldots,N, \; \eta(x_i)=\sign(a_i) }.
\end{align}
where 
\eq{
	\norm{\eta}_{\infty,\Gg}= \max\enscond{ |\eta(t)| }{ t\in \Gg }.
}

We also introduce the corresponding dual problems:
\eql{\label{eq-initial-dualbisdiscr}\tag{$\Dd_\la^\Gg(y_0)$}
	\umin{ \norm{\Phi^* p}_{\infty,\Gg} \leq 1} \left\|\frac{y_0}{\la}-p\right\|_2^2,
}
\eql{\label{eq-constrained-dualdiscr}\tag{$\Dd_0^\Gg(y_0)$}
	\usup{ \norm{\Phi^* p}_{\infty,\Gg} \leq 1} \dotp{y_0}{p}.
}

\begin{rem}
\label{rem-y-imphi}
Let us denote by $G$ the image by $\Phi$ of all measures with support in $\Gg$.
It may happen (for instance if the grid is too rough) that $y_0\notin G$, in which
 case~\eqref{eq-lasso-discr-noiseless} is not feasible and~\eqref{eq-constrained-dualdiscr} has infinite value. But~\eqref{eq-lasso-discr} is then equivalent to $\Pp_\la^\Gg(y_{0,G})$ where $y_0=y_{0,G}+y_{0,G^\perp}$ is an orthogonal decomposition. 
Problem~\eqref{eq-lasso-discr} is thus an approximation of $\Pp_0^\Gg(y_{0,G})$,
 and the relevant dual problems are $\Dd_\la^\Gg(y_{0,G})$ and $\Dd_0^\Gg(y_{0,G})$.
For the sake of simplicity, we shall assume from now on that $y_0\in G$, but the reader may
keep in mind that this hypothesis
 can be withdrawn by replacing $y$ with $y_{0,G}$.
\end{rem}

In view of Remark~\ref{rem-y-imphi}, we observe that problems~\eqref{eq-initial-dualbisdiscr} and 
\eqref{eq-constrained-dualdiscr} are in fact finite dimensional. Indeed, 
 their constraints being invariant by addition of elements of $G^\perp$,
  we may consider their quotient with the space $G^\perp$.
Therefore the condition $p\in L^2(\TT)$ may be reduced to $p\in G$ where $G$ is a finite dimensional space.

As a consequence, a solution to~\eqref{eq-constrained-dualdiscr} always exists,
 so that we may define the \textit{discrete minimal norm certificate}:
 \begin{align}
	\eta_0^\Gg=\Phi^* p_0^\Gg,&  
	\qwhereq
	p_0^\Gg=\uargmin{p}
		\enscond{ \norm{p}_2 }{ p \mbox{ is a solution of }\eqref{eq-constrained-dualdiscr} }.
\label{eq-min-norm-certifdiscr}
\end{align}
The solutions of~\eqref{eq-lasso-discr} and 
\eqref{eq-initial-dualbisdiscr}
(resp.~\eqref{eq-lasso-discr-noiseless} and~\eqref{eq-constrained-dualdiscr})
are related by the extremality conditions~\eqref{eq-extremal-cdt} (resp.~\eqref{eq-extremal-constrained}) where the total variation is replaced with its discrete counterpart $\normTVG{\cdot}$. 

%%%%%%%%%%%%%%%%%%%%%%%%%%%%%%%%%%%%%%%%%%%%%%%%%%%%%%%%%
\subsection{Noise Robustness}
\label{sec-discr-robustness}

As in the continuous case (cf. Section~\ref{sec-noise-robust}), the support of the solutions of $\Pp_\la^\Gg(y_0+w)$ for $\la \to 0^+$ and $\norm{w}_2=O(\la)$ is governed by the minimal norm certificate. We introduce here the discrete counterpart of the extended support of a measure.

\begin{defn}[Extended support]
Let $m_0\in \Mm(\TT)$ such that $y_0=\Phi(m_0)\in G$, and let $\eta^\Gg_0$ 
be the discrete minimal norm certificate defined in~\eqref{eq-min-norm-certifdiscr}.
 The extended support of $m_0$ relatively to $\Gg$ is defined as
\begin{align}
	\extg(m_0) =\enscond{ t\in \Gg }{ |\eta_0^\Gg(t)|= 1 },
\end{align}
and the extended signed support relatively to $\Gg$ as
\begin{align}
	\extsg(m_0) = \enscond{ (t,v) \in \Gg \times \{-1,+1\} }{ \eta_0^\Gg(t)= v }.
\end{align}
\end{defn}

It is important to notice that the assumption $y_0\in G$ does not mean that the support of $m_0$ is included in $\Gg$, but that there exists  a measure with support included in $\Gg$ which produces the same observation $y_0$. Therefore the support
 of $m_0$ and its extended support may even  be disjoint.
 
As in the continuous case, notice that $m_0$ is a solution of~\eqref{eq-lasso-discr-noiseless}
if and only if $\ssupp(m_0) \subset \extsg(m_0)$.

\begin{thm}[Noise robustness, discrete case]
Let $m_0\in \Mm(\TT)$ such that $y_0=\Phi(m_0)\in G$. Then, there exists $\alpha>0, \la_0>0$,
 such that for $(\la,w) \in D_{\alpha,\la_0}$
 (defined in~\eqref{eq-constr-set}) any solution $\tilde{m}_{\la,w}$ of $\Pp_\la^\Gg(y_0+w)$ satisfies:
\begin{align}
 	\ssupp(\tilde{m}_{\la,w}) & \subset \extsg(m_0).
\end{align}
If, in addition, $\Phi_{\extg(m_0)}$ has full rank and $m_0$ is a solution of~\eqref{eq-lasso-discr-noiseless}, then the solution $\tilde{m}_{\la,w}$ is unique, $m_0$ is identifiable and 
choosing $\la = \norm{w}_2/\al$ ensures $\norm{\tilde{m}_{\la,w}-m}_{2,\Gg} = O(\norm{w})$, where 
\eq{
	\norm{\tilde{m}_{\la,w}-m}_{2,\Gg}^2 = \sum_{x\in \Gg} |m(\{x\})-\tilde{m}_{\la,w}(\{x\})|^2.
}
\label{thm-noise-robustness-discr}
\end{thm}

\begin{proof}
The proof is essentially the same as in the continuous case, therefore we only sketch it. To simplify the notation, we write $J=\extg(m_0)$. The solutions of~\eqref{eq-initial-dualbisdiscr} converge to $p_0^\Gg\in L^2(\TT)$ for $\la \to 0^+$, where $\Phi^*p_0^\Gg=\eta_0^\Gg$ is the discrete minimal norm certificate.

By the triangle inequality
\begin{align*}
\normig{\tilde{\eta}_\la - \eta_0^\Gg}&\leq \underbrace{\normig{\tilde{\eta}_\la-\eta_\la^\Gg}}_{\leq C\frac{\norm{w}_2}{\la} } + \normig{\eta_\la^\Gg-\eta_0^\Gg}
\end{align*}
Thus, there exist two constants $\alpha >0$ and $\la_0>0$, such that for $\frac{\norm{w}_2}{\la}\leq \alpha $ and $0<\la<\la_0$, $|\tilde{\eta}_\la(x)|<1$ for any $x\in \Gg\setminus J$.
Then, the primal-dual extremality conditions imply that for any solution $\tilde{m}_{\la,w}$ of $\Pp_\la^\Gg(y_0+w)$, one has $\supp(\tilde{m}_{\la,w}) \subset J$ and equality of the signs. 

Now, if $\Phi_J$ has full rank, we can invert the extremality condition:
\begin{align*}
	\frac{1}{\la}\Phi_J^* \pa{ 
		y_0 +w - \Phi_J( \restr{\tilde{m}_{\la,w}}{J} ) 
	} 
	&= \restr{\tilde{\eta}_\la}{J},\\
\mbox{so that} \quad \restr{\tilde{m}_{\la,w}}{J} = \restr{m}{J} + \Phi_J^+ w-& \la (\Phi_J \Phi_J^*)^{-1}  \restr{ \tilde{\eta}_\la }{J}.
\end{align*}
Observing that $\normig{ \restr{\tilde{\eta}_\la}{J} }\leq 1$, we obtain the $\ell_2$-robustness result.
\end{proof}

Theorem~\ref{thm-noise-robustness-discr} is analogous to Lemma~\ref{lem-robust-boites} for the continuous problem.
The discrete nature of the problem makes its conclusions more precise.
Although the $\ell^2$-robustness results are similar to those of Theorem~\ref{thm-noise-robustness},
the focus here is a bit more general, in the sense that this theorem does not assert that the
support of the recovered measures matches the support of the input measure $m_0$.
In fact, if $m_0$ is a solution to~\eqref{eq-lasso-discr-noiseless-finite}, $\ssupp(m_0)\subset \extsg(m_0)$, so  that the recovered solutions to $\Pp_\la^\Gg(y_0+w)$ have in general more spikes 
than $m_0$, and the spikes in $\extg(m_0)\setminus \supp(m_0)$ must vanish as $\la\to 0, \norm{w}_2 \to 0$.

In order to get the exact recovery of the signed support for small noise, we may assume in addition that
$\ssupp(m_0)=\extsg(m_0)$ so as to obtain a result analogous to Theorem~\ref{thm-noise-robustness}.
Precisely, we obtain the following theorem which was initially proved by Fuchs~\cite{fuchs2004on-sp}.
First, we introduce a pre-certificate.

\begin{defn}[Fuchs pre-certificate]
Let $m_0 \in \Mm(\TT)$ such that $\supp(m_0)\subset \Gg$. We define the Fuchs pre-certificate as
\begin{align}
\etaF =  \uargmin{\eta =\Phi^* p, p \in L^2} \norm{p}
	\quad\text{subject to}\quad
		\restr{\eta}{\supp m_0} = \sign( \restr{m_0}{\supp m_0}).
\label{eq-fuchs-certif}
\end{align}
\end{defn}
This pre-certificate, introduced in \cite{fuchs2004on-sp}, is a certificate for $m_0$ if and only if $\normig{\etaF} \leq 1$, in which case it is equal to the discrete minimal norm pre-certificate $\eta_0^\Gg$.
 
If $\Phi_{\supp m_0}$ has full rank, then $\etaF$ can be computed by solving a linear system:
\eq{
	\etaF = \Phi^* \Phi_{I}^{+,*} \sign( \restr{m}{I}) 	
	\qwhereq
  I=\supp m_0 
  \qandq \Phi_{I}^{+,*} = \Phi_{I} (\Phi_{I}^*\Phi_{I})^{-1}. 
}

\begin{cor}[Exact support recovery, discrete case,\cite{fuchs2004on-sp}]\label{cor-fuchs}
	Let $m_0 \in \Mm(\TT)$ such that $\supp(m_0) \subset \Gg$, and that $\Phi_{\supp m_0}$ has full rank.
	 If $|\etaF(t)|<1$ for all $t\in \Gg\setminus \supp m_0$,
	  then $m_0$ is identifiable for $\Gg$ and there exists $\alpha>0, \la_0>0$, such that for
	   $(\la,w) \in D_{\alpha,\la_0}$ the solution $\tilde{m}_{\la,w}$ of $\Pp_\la^\Gg(y+w)$
	    is unique and satisfies $\ssupp(\tilde{m})=\ssupp(m_0)$. Moreover 
	\eql{\label{eq-expression-discr-explicit}
		\restr{\tilde{m}_{\la,w}}{I} = \restr{m_0}{I} + \Phi_{I}^+ w - \la (\Phi_{I}\Phi_{I}^*)^{-1} \sign(\restr{m_0}{I}),
	}
  where $I=\supp m_0$.
\end{cor}

The condition $|\etaF(t)|<1$ for all $t\in \Gg\setminus \supp m_0$ is often called the irrepresentability
 condition in the statistics literature, see~\cite{Zhao-irrepresentability}. This condition
  can be shown to be almost a necessary and sufficient condition to ensure exact recovery of
   the support of $m_0$. For instance, if $|\etaF(t)|>1$ for some $t\in \Gg\setminus \supp m_0$,
    one can show that for all $\la > 0$ $\supp(\tilde{m}_\la) \neq \supp m_0$ where $\tilde{m}_\la$ is any
     solution of $\Pp_\la^\Gg(y_0)$, see~\cite{vaiter2011robust}.
In our framework, we see that this irrepresentability condition means that the precertificate $\etaF$
 is indeed a certificate (so that it is equal to the minimal norm certificate), and that its 
saturation set is equal to the support of $m_0$. %When $\etaF$ is not a certificate, this means that $\ssupp(m_0) \neq \extsg(m_0)$
%so that the signed support of $m_0$ is not stable.

For deconvolution problems, an important issue is that Corollary~\ref{cor-fuchs} is
 useless when studying the stability of the original infinite dimensional problem~\eqref{eq-initial-pb}.
  Indeed, the pre-certificate~\eqref{eq-fuchs-certif} is not constrained to have vanishing derivatives,
so that it generally takes some values strictly greater than $1$  for a generic discrete input measure $m_0$.
When the stepsize of the grid is small enough, such values are sampled and $\normig{\etaF}$ necessarily becomes strictly larger than one.
  As detailed in Section~\ref{sec-vanishing}, when shifting from the discrete grid setting
   to the continuous setting, the natural pre-certificate to consider is the vanishing derivative pre-certificate $\etaV$ defined in~\eqref{eq-vanishing-der-certif}, and not the pre-certificate $\etaF$.

%%%%%%%%%%%%%%%%%%%%%%%%%%%%%%%%%%%%%%%%%%%%%%%%%%%%%%%%%%%%%%%%%%%%%%%%%
\subsection{Structure of the Extended Support for Thin Grids}
\label{sec-extended}

In the previous section, we have introduced the notion of extended signed support of a measure $m_0$ relatively to a grid $\Gg$,
and we have proved that this set, $\extsg{m_0}$, contains the  signed supports of all the reconstructed measures for small noise.
In this section, we focus on the structure of the extended support. We show that, if the
 support of $m_0$ belongs to the grid for a sufficiently small stepsize and if the Non Degenerate Source Condition holds,
  the extended signed support consists in the signed support of $m_0$ and possibly one immediate neighbor with the same sign for each spike. 
Therefore, when the grid stepsize is small enough, the support of the measure is generally not stable
for the discrete problem, but the support of the reconstructed measure is a close approximation
 of the original one.

From now on, for the sake of simplicity, we consider dyadic grids $\Gg_n=\enscond{ \frac{j}{2^n} }{ 0\leq j \leq 2^n-1 }$. The constraint sets in $\Dd_\la^{\Gg_n}(y_0)$ and~\eqref{eq-initial-dual} are denoted respectively by
\begin{align}
	C_n&=\enscond{ p\in L^2(\TT) }{
		 \left|(\Phi^*p)\left(\frac{j}{2^n}\right)\right|\leq 1,\ 0\leq j\leq 2^n }, \\
	\mbox{and }
	C &= \enscond{ p\in L^2(\TT) }{
			\left\|\Phi^*p\right\|_{\infty}\leq 1
		}
		= \bigcap_{n\in \NN} C_n.
\end{align}

The structure of $\extsg(m_0)$ for large $n$ is intimately related to the convergence of $p^{\Gg_n}_0$ to $p_0$. First, let us notice the following result, whose proof is given in Appendix~\ref{sec-proof2}.

\begin{prop}[Convergence for fixed $\la$]
Let $m_0\in \Mm(\TT)$. Then, for any $\la >0$,
\begin{align}
 \lim_{n\to +\infty} p_\la^{\Gg_n}&= p_\la \mbox{ for the } L^2(\TT) \mbox{(strong) topology},\\
\mbox{ and } \lim_{n\to +\infty} \eta_\la^{\Gg_n}&= \eta_\la \mbox{ for the topology of the uniform convergence}.
\end{align}
Moreover, if there exists a solution to the continuous dual problem~\eqref{eq-constrained-dual},
\begin{align}
 \lim_{\la \to 0^+} \lim_{n\to +\infty} p_\la^{\Gg_n}= p_0, \mbox{ and } \lim_{\la \to 0^+} \lim_{n\to +\infty} \eta_\la^{\Gg_n}= \eta_0. 
 \label{eq-double-limite}
\end{align}
\label{prop-cv-fixedla}
\end{prop}

Proposition~\ref{prop-cv-fixedla} simply states that the projection onto convex sets $C_n$ which
converge (in the sense of set convergence) to $C$ converges to the projection onto $C$. However, 
the case $\la=0$ is not as straightforward, and for instance one cannot easily swap the limits in~\eqref{eq-double-limite}.
In fact, given any decreasing sequence of polyhedra $C_n$, it is not true in general that 
the minimal norm solution of $\sup_{p\in C_n} \dotp{y_0}{p}$ should converge to
the minimal norm solution of  $\sup_{p\in C} \dotp{y_0}{p}$ where $C=\bigcap_{n\in \NN}C_n$.
 As a consequence it is not clear to us whether this convergence always holds for polyhedra of the form 
\eq{
 	C_n = \enscond{ p\in L^2 }{ \norm{\Phi^*p}_{\infty,\Gg_n}\leq 1 }.
}

However, when the spikes locations belong to the grid for $n$ large enough, the convergence of the minimal norm certificates holds.
In the case of dyadic grids, this is equivalent to $m_0\in \Mm(\TT)$ being a \textit{discrete dyadic measure}, i.e. such that for some $n_0\in \NN$:
\begin{align}
	m = \sum_{i=1}^N a_{i} \delta_{x_{i}}, 
	\qwithq 
	x_i=\frac{j_i}{2^{n_0}} \mbox{ and } 0\leq j_i \leq 2^{n_0}-1.
\label{eq-dyadic-measure}
\end{align}

The proofs given below make use of a remark given in~\cite{Candes-toward}: if a solution of the continuous problem~\eqref{eq-constrained-pbm} has support in the grid $\Gg$, then it is also a solution of the discrete problem~\eqref{eq-lasso-discr-noiseless}.

\begin{prop}[Convergence for dyadic measures]
Let $m_0\in \Mm(\TT)$ be a discrete dyadic measure (see~\eqref{eq-dyadic-measure}), and assume
 that the (possibly degenerate) source condition holds. Then
\begin{align}
\lim_{n\to +\infty} p_{0}^{\Gg_n}&=p_0 \mbox{ for the } L^2 \mbox{ (strong) topology},\\
\mbox{and }\lim_{n\to +\infty} \eta_{0}^{\Gg_n(i)}&=\eta_0^{(i)} \mbox{ for } 0\leq i \leq 2, \mbox{ in the sense of the uniform convergence,}
\end{align}
where $\eta_0 = \Phi^* p_0$ (resp. $\eta_{0}^{\Gg_n}=\Phi^* p_{0}^{\Gg_n}$) denotes the corresponding minimal norm certificate.
\label{prop-cv-dyadic}
\end{prop}

\begin{proof}
First, following~\cite{Candes-toward}, we observe that, since $(\Phi^*p_0)(x_i)=\sign(a_i)$ and $\normi{\Phi^*p_0} \leq 1$ (a fortiori $|\Phi^*p_0\left(\frac{j}{2^n}\right)|\leq 1$ for $1\leq j\leq 2^n-1$),
$\Phi^*p_0$ is also a dual certificate for~$(\Pp_0^{\Gg_n})$ provided $n\geq n_0$. As a consequence $\norm{p_{0}^{\Gg_n}}_2\leq \norm{p_0}_2$.

The sequence $(p_{0}^{\Gg_n})_{n\in \NN}$ being bounded in $L^2(\TT)$, we may extract a subsequence (still denoted by $p_{0}^{\Gg_n}$) 
 which weakly converges to some $\tilde{p}\in L^2(\TT)$, and
\begin{align}
 	\norm{\tilde{p}}_2 \leq \liminf_{n\to +\infty} \norm{p_{0}^{\Gg_n}}_2 
	\leq \limsup_{n\to +\infty} \norm{p_{0}^{\Gg_n}}_2 \leq \norm{p_0}_2.
 	\label{eq-limsup-norm}
\end{align}
Moreover, by optimality of $p_{0}^{\Gg_n}$ for the discrete problem, for each $p\in C\subset C_n$,
 $\dotp{y_0}{p_{0}^{\Gg_n}} \geq \dotp{y_0}{p}$ so that in the limit $\dotp{y_0}{\tilde{p}} \geq \dotp{y_0}{p}$.
Observing that $\tilde{p}\in C=\bigcap_{n\in \NN}{C_n}$ (since each $C_n$ is
  weakly closed) we conclude that $\tilde{p}=p_0$. Since the limit does not depend on the extracted subsequence,
  we conclude that the whole sequence $(p_0^{\Gg_n})_{n\in \NN}$ converges to $p_0$, and equality in~\eqref{eq-limsup-norm}
  implies that the convergence is strong.

The consequence regarding $\eta_{0}^{\Gg_n}$ is straightforward.
\end{proof}

We may now describe the structure of the extended support for dyadic measures which satisfy the Non Degenerate Source Condition.

\begin{prop}[Extended support]
Let $m_0=\sum_{i=1}^N a_i \delta_{x_i}$ be a discrete dyadic measure which satisfies the Non Degenerate Source Condition.
Then, for $n$ large enough, there exists $\varepsilon^n \in \{+1,-1\}^N$
such that:
\begin{align}
\ssupp(m_0) \subset \exts_{\Gg_n}(m_0) \subset \ssupp(m_0) \cup \left(\ssupp(m_0) + \frac{\varepsilon^n}{2^n} \right), 
\end{align}
where $\ssupp(m_0) + \frac{\varepsilon^n}{2^n}= \enscond{ (x_i+\frac{\varepsilon^n_i}{2^n},\eta_{0}^{\Gg_n}(x_i)) }{ 1\leq i\leq N }$.
\label{prop-extsupp}
\end{prop}

\begin{cor}
Under the hypotheses of Proposition~\ref{prop-extsupp}, for $n$ large enough, there exist two constants $\alpha(n)>0$
 and $\lambda_0(n)>0$ such that, for $\frac{\norm{w}_2}{\la}<\alpha(n)$ and  $0<\la <\la_{0}(n)$,
 any solution $\tilde{m}_\la^{\Gg_n}$ of $(\Pp_\la^{\Gg_n})$ has support in $\{ x_i,\ 1\leq i\leq N\} \cup \{x_i+\frac{\varepsilon^n_i}{2^n}, \ 1\leq i\leq N \}$, with signs $\eta_{0}^{\Gg_n}(x_i)$, $1\leq i\leq N$.
\label{cor-extended-support}
  \end{cor}
    
  \begin{proof}[Proof of Proposition~\ref{prop-extsupp}]
We describe the points where the value of $\eta_0^{\Gg_n}$ may be $\pm 1$.
By the Non-Degenerate Source Condition, there exists $\epsilon>0$ small enough such that
 the intervals $(x_{0,i}-\epsilon, x_{0,i}+\epsilon)$, $1\leq i\leq N$, do not intersect,
 and that for all $t\in \bigcup_{i=1}^N (x_{i}-\epsilon, x_{i}+\epsilon)$, $|\eta_0(t)|\geq C>0$ and $|\eta_0''(t)|\geq C>0$.
Moreover, $\sup_{K_\epsilon} |\eta_0| <1$ with $K_\epsilon = \TT\setminus \bigcup_{i=1}^N (x_{i}-\epsilon, x_{i}+\epsilon)$.
 
 Therefore, by Proposition~\ref{prop-cv-dyadic}, for $n$ large enough: 
 \begin{itemize}
 \item $|\eta_{0}^{\Gg_n}(t)|\geq \frac{C}{2}>0$ for $t\in (x_{i,0}-\varepsilon,x_{i,0}+\varepsilon)$,
 \item $|(\eta_{0}^{\Gg_n})''(t)|\geq \frac{C}{2}>0$ for $t\in (x_{i}-\varepsilon,x_{i}+\varepsilon)$,
 \item $\sup_{K_\epsilon} |\eta_{0}^{\Gg_n}| <1$,
 \end{itemize} 
and in each interval $(x_{i,0}-\varepsilon,x_{i}+\varepsilon)$, $\eta_{0}^{\Gg_n}$ has the same sign as $\eta_{0}$
 and it is strictly concave (resp. strictly convex) if $\eta_0(x_i)=1$ (resp. $-1$).
  
Assume for instance that $\eta_0(x_i)=1$. The extremality conditions between $p_0$ and $m_0$ for $(\Pp_0(y))$ also imply that 
$m_0$ is a solution of~$(\Pp_0^{\Gg_n}(y_0))$. Then, the extremality conditions
 between $p_{0}^{\Gg_n}$ and $m_0$ imply that $\eta_{0}^{\Gg_n}(x_{i})=1$ as well.
By the strict concavity of $\eta_{0}^{\Gg_n}$ there is at most one other point $t^\star\in (x_{i}-\varepsilon,x_{i}+\varepsilon)$
 such that $\eta_{0}^{\Gg_n}(t^\star)=1$, and since $\eta_{0}^{\Gg_n}(x_{i}\pm \frac{1}{2^n})\leq 1$, $|t^\star-x_i|\leq \frac{1}{2^n}$.
Such a point $t^\star$ contributes to the extended support of $m$ if and only if it belongs to the grid (i.e. $t^\star=x_i\pm \frac{1}{2^n}$).

The argument for $\eta_0(x_i)=-1$ is similar. This concludes the proof.
  \end{proof}

Corollary~\ref{cor-extended-support} highlights the difference between the continuous and 
the discretized problems. In the first case,
any small noise would induce a slight perturbation of the spikes locations and amplitudes,
 but their number would stay the same. In the second case, the spikes cannot ``move'', so that
 new spikes may appear, but only at one of the immediate neighbors of the original ones.

For non-dyadic measures, we may show using Proposition~\ref{prop-cv-fixedla} that for small,
 fixed $\la >0$, and $n$ large enough, there is at most one pair of spikes (located at
  consecutive points of the grid) in the neighborhood
  of each original spike. From our numerical experiments described below (in the case of the ideal low-pass filter),
 we conjecture that, in the case where there are indeed two spikes, they surround the location of the original spike.

\subsection{Application to the Ideal Low-pass Filter}

To conclude this section, we compare the different (pre-)certificates involved in the above discussion,
 whether on the discrete grid or in the continuous domain. 
Then we illustrate the convergence of the sets $(C_n)_{n\in \NN}$ towards $C$.

\paragraph{Certificates.}
Figure~\ref{fig-certif-grid} illustrates the results of Section~\ref{sec-extended}.
The numerical values are $f_c=6$, $n=7$, and the distance between the two opposite spikes is $\frac{0.6}{f_c}$.
The continuous minimal norm certificate $\eta_0$ is shown: it satisfies $|\eta_0(t)|\leq 1$
 for all $t\in \TT$ and $\eta_0(x_i)=\sign m_0(\{x_i\})$ for $1\leq i\leq N$. The discrete
  minimal norm certificate $\eta_0^{\Gg_n}$ satisfies $|\eta_0^{\Gg_n}(t)|\leq 1$
 for all $t$ in the grid, and $\eta(u)=\sign m_0(\{x_i\})$ for all $u\in\ext_{\Gg_n}$
in the neighborhood of $x_i$. For a dyadic measure, such points are $x_i$ and possibly one of its immediate neighbors.
 For non dyadic measures, we conjecture that such points are the two immediate neighbors of $x_i$.

The Fuchs precertificate $\etaF$ is also shown. Some points $t$ of the grid do not satisfy $|\etaF(t)|\leq 1$, hence the Fuchs pre-certificate is not a certificate and the support is not stable. This was already clear from the fact that $\supp(m_0) \subsetneq \ext_{\Gg_n}(m_0)$.

\begin{figure}[htbp]
\centering
\setcounter{subfigure}{0}
\subfloat[Dyadic measure]{\includegraphics[width=0.47\linewidth] {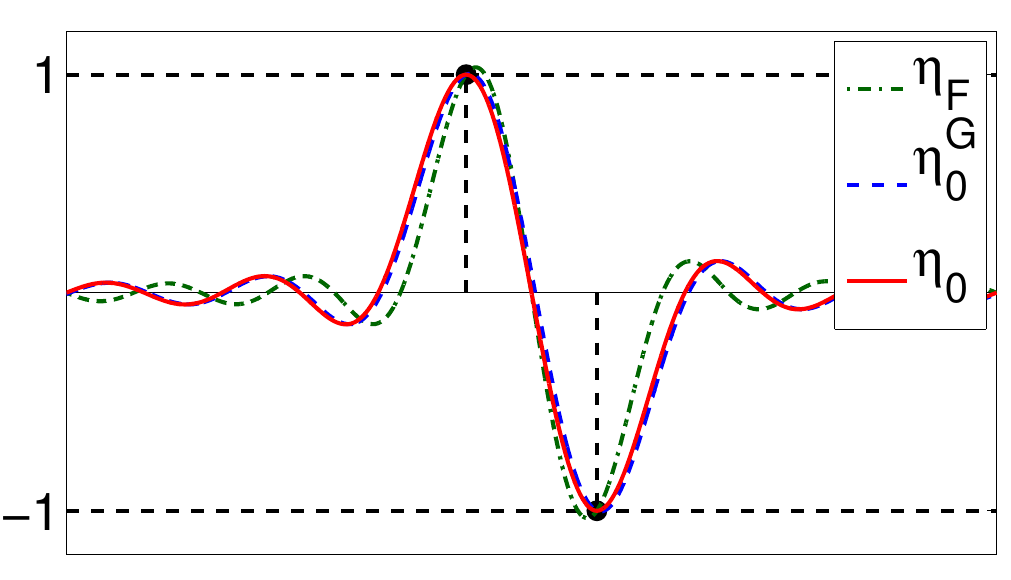}}
\subfloat[Non dyadic measure]{\includegraphics[width=0.47\linewidth] {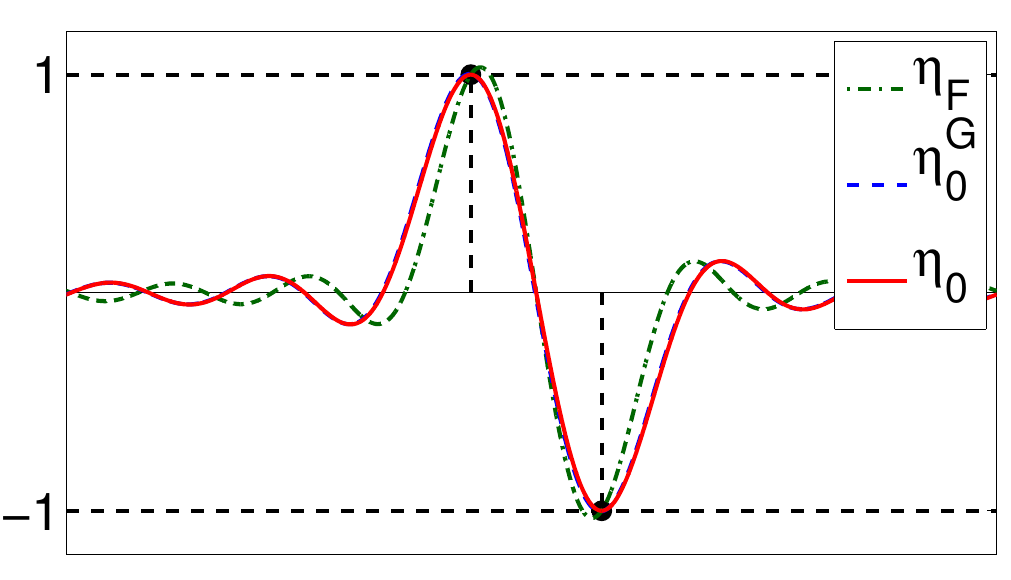}}\\
\subfloat[Zoom]{\includegraphics[width=0.45\linewidth] {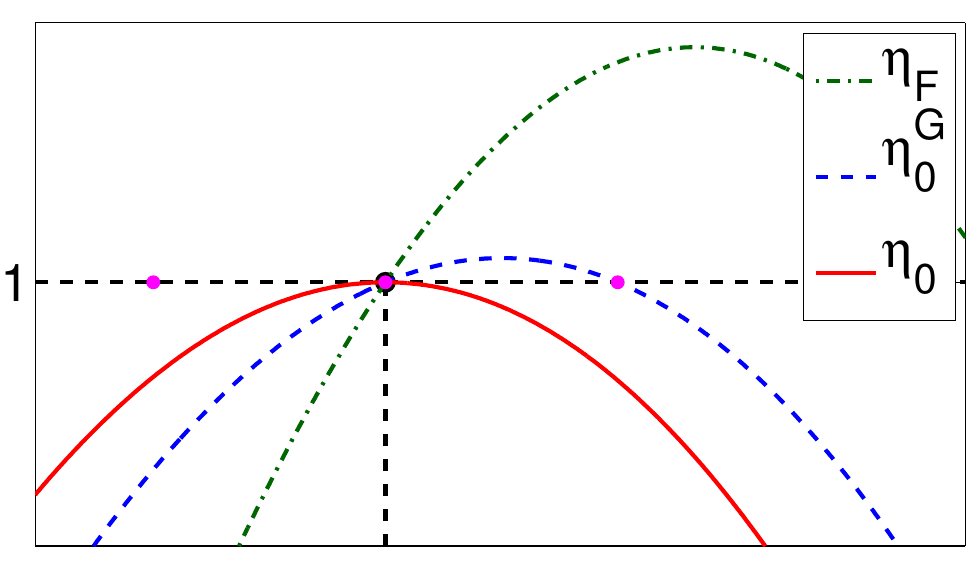}}
\hspace{2mm}\subfloat[Zoom]{\includegraphics[width=0.45\linewidth] {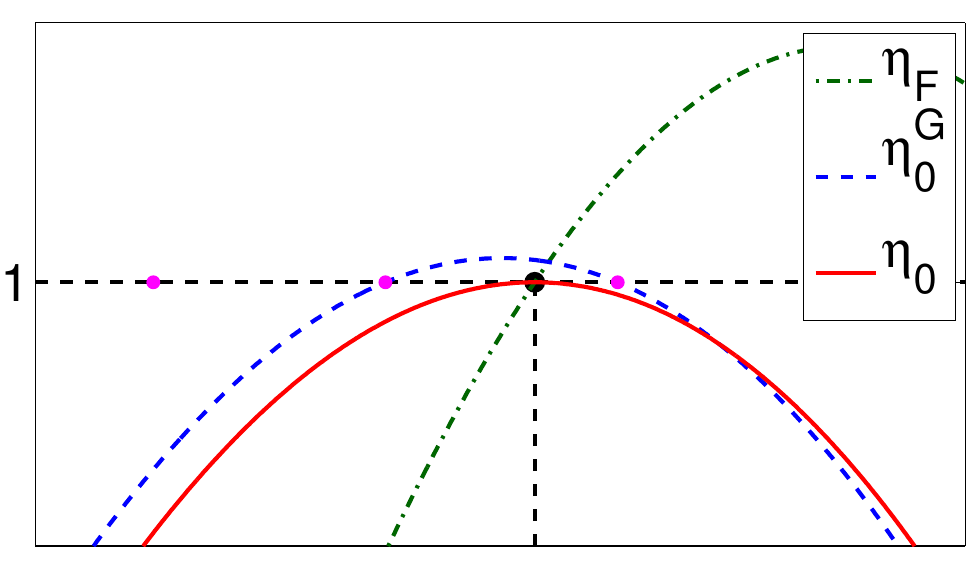}}\\
\caption{\label{fig-certif-grid} Comparison of certificates for a dyadic (left) and a non-dyadic measure (right).
The second row is a zoom of the first one near the left spike.
The (continuous) minimum norm certificate $\eta_0$ (in continuous red line) is everywhere bounded by $1$.
The (discrete) minimum norm certificate $\eta_0^{\Gg_n}$ (in dashed blue line) is bounded by $1$ at the grid points.
The Fuchs pre-certificate $\etaF$ (dash-dot green line) is above $1$ at some points of the grid: the Fuchs criterion is not satisfied.
}
\end{figure}

Figure~\ref{fig-coeff-grid} focuses on the reconstructed amplitudes $\tilde{a}_i$ using $\Pp_\la(y_0)$ as $\la \to 0$. Each curve represents a path $\la \mapsto \tilde{a}_i$. Note that for the problem on a finite grid, such paths are piecewise affine. In the dyadic case (left part of the figure), the amplitude at $x_i$ (continuous line) and at the next point of the grid (dashed line) are shown.
As $\la \to 0$, the spike at the neighbor vanishes and the result tends to the original
 identifiable measure.
In the non dyadic case (right part of the figure), the amplitude at the two immediate neighbors of $x_i$ are shown (continuous and dashed lines).
Here $\supp m_0 \not\subset \Gg$ so that $m_0$ is not identifiable for the discrete problem.
For each spike, the amplitudes of the two neighbors converge to some non zero value. The limit measure as $\la \rightarrow 0$ is the solution of $\Pp_0(y_{0G})$.

\begin{figure}[htbp]
\centering
\subfloat[Dyadic measure]{\includegraphics[width=0.48\linewidth] {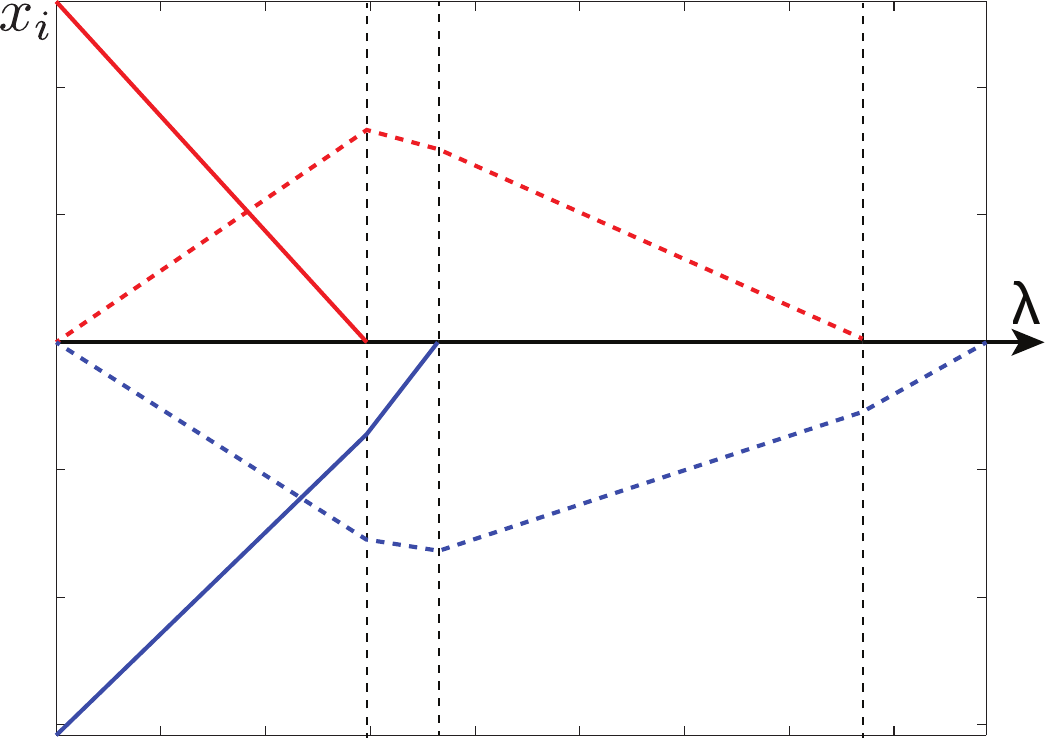}}
\subfloat[Non dyadic measure]{\includegraphics[width=0.48\linewidth] {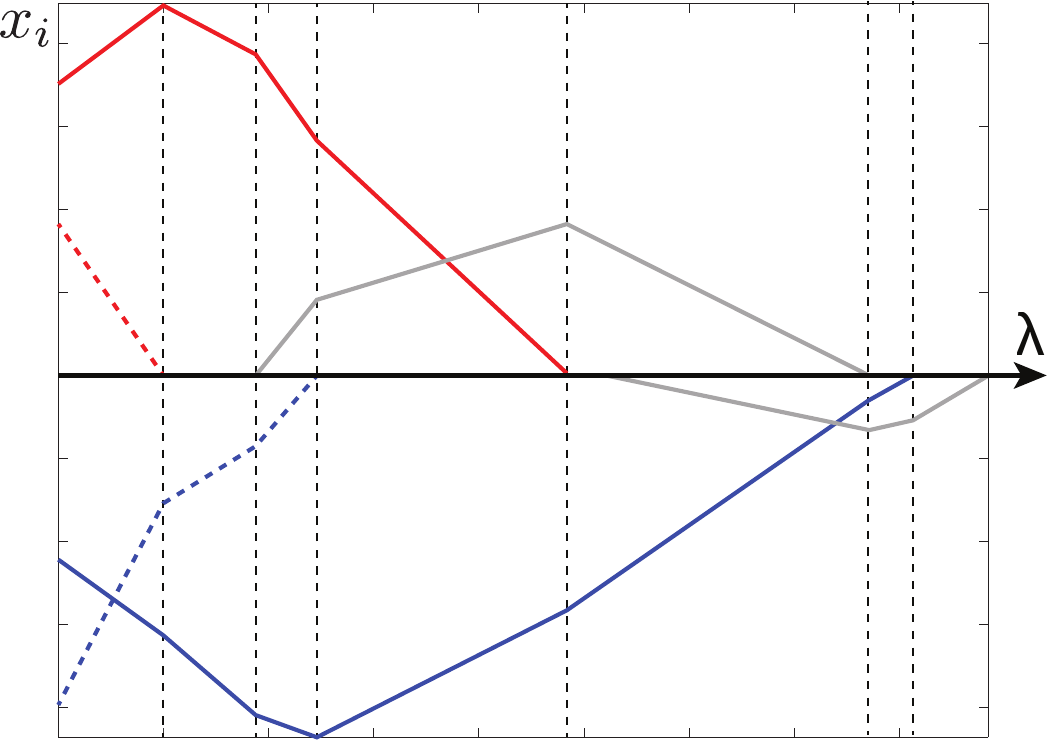}}
\caption{\label{fig-coeff-grid} 
Display of the solution path (as a function of $\la$) for the measure displayed on Figure~\ref{fig-certif-grid}. 
Left: Amplitudes of the coefficients at $x_i$ (continuous line) and at the next point of the grid (dashed line) as $\la$ varies.
Right: idem for the two immediate neighbors of $x_i$. 
Some other spikes (grey continuous line) appear and vanish before the last segments, as $\la \to 0$.}
\end{figure}

\paragraph{Set convergence.} Now, we interpret the convergence of the discrete problems
 through the convergence of the corresponding constraint set for the dual problem.
Writing $\Phi^* p(x)=\int p(t)\phi(x-t) \d t= \dotp{p}{\phi_x}_{L^2}$ with $\phi_x : t\mapsto \phi(x-t)$,
 we observe that:
\begin{align}
 	C_n &= \enscond{ p\in \Im \Phi }{
 		\left|\Phi^*p \left(\frac{j}{2^n}\right)\right|\leq 1, \ 0\leq j \leq 2^n-1
	}\\
 		&= \enscond{p\in \Im \Phi }{ 
		|\dotp{p}{\phi_{\frac{j}{2^n}}}_{L^2}| \leq 1, \ 0 \leq j \leq 2^n-1
	}.
 \end{align}
 As a consequence $C_n$ is the polar set of the convex hull of 
$\enscond{ \pm \phi_{\frac{j}{2^n}} }{ 0\leq j \leq 2^n-1 }$.

In the case of the Dirichlet kernel, the vector space $\Im \Phi$ is the space of trigonometric polynomials with degree less than or equal to $f_c$.
An orthonormal basis of $\Im \Phi$ is given by: $(c_0,c_1,\ldots c_{f_c}, s_1, \ldots s_{f_c})$ where 
$c_0 \equiv 1$, $c_k: t\mapsto \sqrt{2}\cos (2\pi kt)$ and $s_k: t\mapsto \sqrt{2}\sin (2\pi kt)$ for $1\leq k\leq f_c$.

Moreover,
\begin{align*}
\phi(x-t)&= \frac{1}{2f_c+1}\left(1+\sum_{k=1}^{f_c} 2\cos (2\pi k (x-t))\right)\\
&= \frac{1}{2f_c+1}\left(1+2\sum_{k=1}^{f_c} \left(\cos (2\pi kx) \cos(2\pi kt) +\sin (2\pi kx)\sin (2\pi kt) \right)\right)
\end{align*}
so that we may write:
\begin{align*}
\phi_x=\frac{1}{2f_c+1}\left(c_0+\sqrt{2}\sum_{k=1}^{f_c}\left(\cos (2\pi kx) c_k + \sin (2\pi kx)s_k \right)\right).
\end{align*}

\begin{figure}[htbp]
\centering
\subfloat{\includegraphics[width=0.30\linewidth,clip=true,trim=90px 40px 90px 40px]{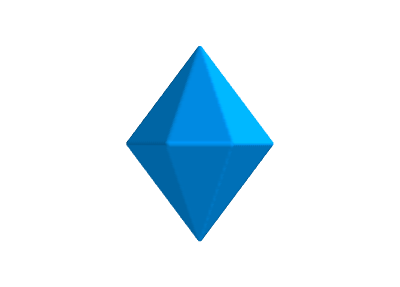}}\hfill
\subfloat{\includegraphics[width=0.30\linewidth,clip=true,trim=90px 40px 90px 40px]{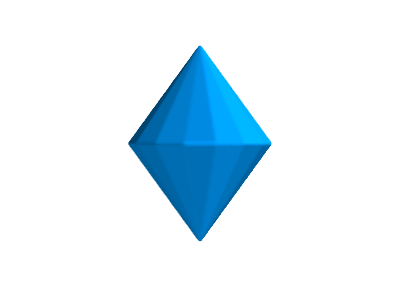}}\hfill
\subfloat{\includegraphics[width=0.30\linewidth,clip=true,trim=90px 40px 90px 40px]{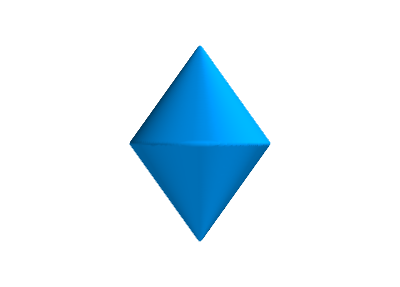}}\\    \setcounter{subfigure}{0}
\subfloat[$C_3$]{\includegraphics[width=0.30\linewidth,clip=true,trim=90px 40px 90px 40px]{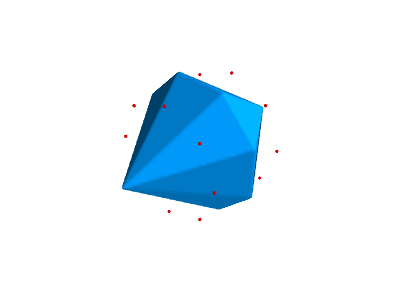}}\hfill
\subfloat[$C_4$]{\includegraphics[width=0.30\linewidth,clip=true,trim=90px 40px 90px 40px]{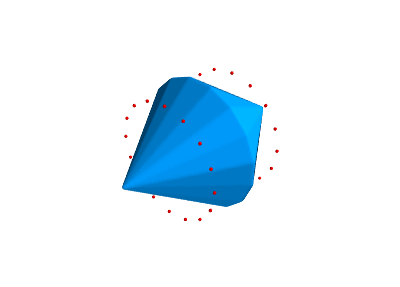}}\hfill
\subfloat[$C_7$]{\includegraphics[width=0.30\linewidth,clip=true,trim=90px 40px 90px 40px]{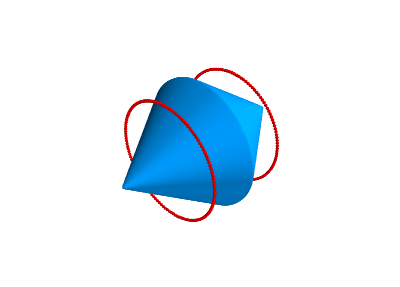}}\\
\caption{\label{fig-convexe} Top: The convex set $C_n$ for $f_c=1$, and $n=3$, $4$ or $7$ (from left to right).
Bottom: same convex sets, the red spheres indicate the (rescaled) vectors $\phi_{\frac{j}{2^n}}$. }
\end{figure}

For $f_c=1$, we obtain $\phi_x=\frac{1}{3}\left(c_0+\sqrt{2}\left(\cos (2\pi x) c_1 + \sin (2\pi x)s_1 \right)\right)$, and the vectors $\phi_x$ lie on a circle. The convex hull of $\enscond{ \pm \phi_{\frac{j}{2^n}} }{ 0\leq j \leq 2^n-1 }$
 is thus a cylinder, and its polar set $C_n$ is displayed in Figure~\ref{fig-convexe} for $n=3$, $4$, and $7$.

Problem~$(\Dd_\la^{\Gg_n}(y_0+w))$ corresponds to the projection of $\frac{y_0+w}{\la}$ onto the polytope $C_n$. Each face of $C_n$ corresponds to a possible signed support of the solutions $\tilde{m}_{\la,w}$. The large, flat faces of $C_n$ yield stability to the support of $\tilde{m}_{\la,w}$ for
 small noise $w$, as described by Theorem~\ref{thm-noise-robustness-discr}. As $n\to +\infty$ 
these faces converge into a piecewise smooth manifold and the support of $\tilde{m}_{\la,w}$
 is allowed to vary smoothly in $\TT$, according to Theorem~\ref{thm-noise-robustness}.

% !TEX root = ../DuvalPeyre-SparseSpikes.tex

%%%%%%%%%%%%%%%%%%%%%%%%%%%%%%%%%%%%%%%%%%%%%%%%%%%
\section*{Conclusion}

% Whereas~\cite{Candes-superresol-noisy} have provided an $L^2$ robustness to noise result for  the super-resolution problem and~\cite{Fernandez-Granda-support,Azais-inaccurate} have provided general bounds on the errors on the support and the amplitudes of the masses,

In this paper, we have given a precise statement about the support recovery property of sparse spikes deconvolution with total variation regularization.
This support recovery is governed by the non-degeneracy of a minimal norm certificate. This hypothesis can be checked by computing a vanishing derivative pre-certificate, which can be computed in closed form.
We have shown that under this non-degeneracy hypothesis, one recovers the same number of spikes and that these spikes converge to the original ones when $\la$ and $\norm{w}/\la$ are small enough. 
While previous stability results~\cite{Candes-superresol-noisy,Fernandez-Granda-support,Azais-inaccurate} hold for an arbitrary noise level and make use of any non-degenerate certificate, they are formulated in terms of local averages of the recovered measure and do not describe precisely the support. In contrast, our result which requires a specific certificate to be non-degenerate and a regime where $\la$ and $\norm{w}/\la$ are small enough provides exact support stability. These settings and results are thus not comparable, and provide complementary informations about the performance of total variation regularization. 
 
Developing a similar framework for the discrete $\ell^1$ setting, we have also improved upon existing results about stability of the support by introducing the notion of extended support of a measure. Our study highlights the difference between the continuous and the discrete case: when the size of the grid is small enough, the stable recovery of the support is generally not possible in the discrete framework. Yet,  in the non degenerate case, the reconstructed support at small noise is a slight modification of the original one: each original spike yields at most one pair of consecutive spikes which surround it.
  
Finally, let us note that the proposed method extends to non-stationary filtering operators and to arbitrary dimensions.

%%%%%%%%%%%%%%%%%%%%%%%%%%%%%%%%%%%%%%%%%%%%%%%%%%%
\section*{Acknowledgements} 

The authors would like to thank Jalal Fadili, Charles Dossal and Samuel Vaiter for fruitful discussions. This work has been supported by the European Research Council (ERC project SIGMA-Vision).

% !TEX root = ../DuvalPeyre-SparseSpikes.tex

\appendix
\section{Auxiliary results}
\label{sec-auxiliary}

For the convenience of the reader, we give here the proofs of several auxiliary results
 which are needed in the discussion.
 
%%%%%%%%%%%%%%%%%%%%%%%%%%%%%%%%
% \subsection{Strong duality for the constrained problem}

\begin{prop}[Subdifferential of the total variation]
Let us endow $\Mm(\TT)$ with the weak-* topology and $C(\TT)$ with the weak topology.
Then, for any $m\in \Mm(\TT)$, we have:
\begin{align*}
 	\partial \normTV{m} = \enscond{\eta\in C(\TT)}{ \normi{\eta} \leq 1 \qandq \int \eta \, \d m =\normTV{m}  }.
\end{align*}
\label{prop-subdifferential}
\end{prop}

\begin{proof}
Let $A=\enscond{ \eta \in C(\TT) }{ \forall m\in \Mm(\TT), \ \dotp{\eta}{m} \leq \normTV{m} }$.
It is clear that $B_\infty(0,1)\subset A$, where $B_\infty(0,1)$ is the $L^\infty(\TT)$ closed unit ball.
Conversely, we observe that $A\subset B_\infty(0,1)$ by considering the Dirac masses $(\pm \delta_t)_{t\in \TT}$.

Let us write $J(m):= \normTV{m}$. The function $J:\Mm(\TT)\rightarrow \RR\cup \{+\infty\}$ is convex, proper, lower semi-continuous (for the weak-* topology), positively homogeneous and:
\begin{align*}
J^*(\eta)&=\sup_{m\in \Mm(\TT)} \sup_{t>0} \left( \dotp{\eta}{t m} -J(t m) \right)\\
&=\sup_{t>0} t\left( \sup_{m\in \Mm(\TT)} \dotp{\eta}{m} -J(m) \right)\\
&=\choice{
		 0 \quad \mbox{if } \eta\in A,  \\
		 +\infty \quad \mbox{otherwise.} \\
	}
\end{align*}

By Proposition~I.5.1 in \cite{ekeland1976convex}, for any $\eta\in C(\TT)$:
\begin{align*}
\eta\in \partial J(m) \Longleftrightarrow  \dotp{\eta}{m} = J(m) +J^*(\eta),
\end{align*}
which is equivalent to $\normi{\eta}\leq 1$ and $\int \eta \d m = \normTV{m}$.
\end{proof}

\begin{prop}
There exists a solution to~\eqref{eq-constrained-pbm} and the strong duality holds between
 \eqref{eq-constrained-pbm} and \eqref{eq-constrained-dual}, i.e. 
\begin{align}
	\umin{\Phi(m)=y_0} \normTV{m} =	\usup{\normi{\Phi^* p}\leq 1} \dotp{y_0}{p}.
\end{align}
Moreover, if a solution $p^\star$ to \eqref{eq-constrained-dual} exists, 
\begin{align}
	\Phi^* p^\star \in \partial{\normTV{m^\star}}
\label{eq-extremal-dual}
\end{align}
where $m^\star$ is any solution to \eqref{eq-constrained-pbm}.
Conversely, if \eqref{eq-extremal-dual} holds, then $m^\star$ and $p^\star$ are solutions
of respectively \eqref{eq-constrained-pbm} and \eqref{eq-constrained-dual}.
\label{prop-strong-dual}
\end{prop}

\begin{proof}
We apply \cite[Theorem~II.4.1]{ekeland1976convex} to \eqref{eq-constrained-dual} (and not to \eqref{eq-constrained-pbm} as would be natural)
rewritten as
\begin{align*}
\inf_{\normi{\Phi^* p} \leq 1} \dotp{-y_0}{p},
\end{align*}
The infimum is finite since for any admissible $p$, $\dotp{-y_0}{p}=\dotp{m_0}{\Phi p} \geq - \normTV{m_0}$.
Let $V=L^2(\TT)$, $Y=C(\TT)$ (endowed with the strong topology), $Y^*=\Mm(\TT)$, $F(u)=\dotp{-y_0}{u}$ for $u\in V$,
 $G(\psi)=\iota_{\normi{\cdot}\leq 1}(\psi)$ for $\psi \in Y$ and $\Lambda=\Phi^*$. It is clear that $F$ and $G$ are proper convex lower semi-continuous functions. Eventually, $F$ is finite at $0$, G is finite and continuous at $0=\Lambda 0$.
Hence the result.
\end{proof}

%%%%%%%%%%%%%%%%%%%%%%%%%%%%%%%%%%%%%%%%%%%%%%%%%%%%%%%%%%%%%%%%%%%%%%%%%%%%%%%%
\section{Proof of Proposition~\ref{prop-gamma-injective}}
\label{sec-proof1}

Assume that for some $(u,v)\in \RR^N\times \RR^N$, $\Ga_x (u,v)=0$. Then
\begin{align*}
	\foralls  t\in \TT, \quad 0&=\sum_{j=1}^N  \left(u_j\varphi(t-x_j)+v_j \varphi'(t-x_j)\right)\\
&=\sum_{k=-f_c}^{f_c} \left(\sum_{j=1}^N (u_j + 2ik\pi v_j)e^{-2ik\pi x_j}  \right) e^{2ik\pi t}
\end{align*}
We deduce that 
\eq{
	\foralls k\in \{-f_c,\ldots f_c \}, \quad
	\sum_{j=1}^N (u_j + k \tilde{v}_j)r_j^k =0
	\qwhereq
	\choice{
		r_j=e^{-2i\pi x_j},\\
		\tilde{v}_j=2i\pi v_j.
	}
}
It is therefore sufficient to prove that the columns of the following matrix are linearly independent
\begin{align*}
\begin{pmatrix}
r_1^{-f_c} & \ldots &r_N^{-f_c} & (-f_c)r_1^{-f_c}& \ldots &(-f_c)r_N^{-f_c} \\
\vdots & &\vdots  & \vdots & &\vdots \\
r_1^{k} & \ldots &r_N^{k} & kr_1^{k}& \ldots &k r_N^{k} \\
\vdots & &\vdots  & \vdots & &\vdots \\
r_1^{f_c} & \ldots &r_N^{f_c} & (f_c)r_1^{f_c}& \ldots &(f_c)r_N^{f_c} \\
\end{pmatrix}.
\end{align*}
% Since we have not found this result in standard textbooks, we detail the argument.
If $N<f_c$, we complete the family $\{r_1, \ldots r_N\}$ in a family $\{r_0,r_1,\ldots r_{f_c}\}\subset \SS^1$ such
 that the $r_i$'s are pairwise distinct. We obtain a square matrix $M$ by inserting the corresponding columns
\begin{align*}
M = \begin{pmatrix}
r_1^{-f_c} & \ldots &r_{f_c}^{-f_c} & r_0^{-f_c} &(-f_c)r_1^{-f_c}& \ldots &(-f_c)r_{f_c}^{-f_c} \\
\vdots & &\vdots  &\vdots & \vdots & &\vdots \\
r_1^{k} & \ldots &r_{f_c}^{k} & r_0^k &kr_1^{k}& \ldots &k r_{f_c}^{k} \\
\vdots & &\vdots &\vdots & \vdots & &\vdots \\
r_1^{f_c} & \ldots &r_{f_c}^{f_c} & r_0^{f_c}&(f_c)r_1^{f_c}& \ldots &(f_c)r_{f_c}^{f_c} \\
\end{pmatrix}.
\end{align*}
 We claim that $M$ is invertible. Indeed, if there exists $\alpha\in \CC^{(2f_c+1)}$ such that
  $M^T \alpha=0$, then the rational function $F(z) = \sum_{k=-f_c}^{f_c} \alpha_k z^{k}$ satisfies:
\begin{align*}
  	F(r_j)&=0 \qandq F'(r_j)=0 \ \mbox{ for } 1\leq j \leq f_c,\\
  	F(r_0)&=0.
\end{align*}
Hence, $F$ has at least $2f_c+1$ roots in $\SS^1$, counting the multiplicities. This imposes that $F=0$, thus $\alpha=0$,
 and $M$ is invertible. The result is proved.

%%%%%%%%%%%%%%%%%%%%%%%%%%%%%%%%%%%%%%%%%%%%%%%%%%%%%%%%%%%%%%%%%%%%%%%%%%%%%%%%
\section{Proof of Proposition~\ref{prop-cv-fixedla}}
\label{sec-proof2}

Let us denote by $P_{C_n}(x)$ the projection of $x\in L^2(\TT)$ onto $C_n$. We have:

\begin{align*}
\left\| P_{C_n}(\frac{y_0}{\la})-P_{C_n}(0)\right\|_2 \leq \left\| \frac{y_0}{\la} - 0\right\|_2,
\end{align*}
so that the sequence $p_{\la}^{\Gg_n}=P_{C_n}(\frac{y_0}{\la})$ is bounded in $L^2(\TT)$, and we may extract a subsequence $p_{\la}^{\Gg_n'}$
which weakly converges to some $p_\la^\star\in L^2(\TT)$. Since $C_{n'}$ is (weakly) closed for all $n'$, $p_\la^\star\in \bigcap_{n'} C_{n'}=C$.

Moreover, by the characterization of the projection onto convex sets:
\begin{align*}
\forall z\in C\subset C_n', \ \left\langle \frac{y_0}{\la} - p_{\la}^{\Gg_n'},z\right\rangle  - \left\langle \frac{y_0}{\la},p_{\la}^{\Gg_n'}\right\rangle +\norm{p_{\la}^{\Gg_n'}}_2^2 &\leq 0.\\
\mbox{ Passing to the limit } n'\to +\infty, \ \left\langle \frac{y_0}{\la} - p_{\la}^\star,z\right\rangle - \left\langle \frac{y_0}{\la},p_{\la}^\star\right\rangle + \liminf_{n'} \norm{p_{\la}^{\Gg_n'}}_2^2 &\leq 0,\\
 \left\langle \frac{y_0}{\la} - p_{\la}^\star,z\right\rangle  - \left\langle \frac{y_0}{\la},p_{\la}^\star\right\rangle + \norm{p_{\la}^\star}_2^2 &\leq 0,\\
 \left\langle \frac{y_0}{\la} - p_\la^\star,z - p_\la^\star\right\rangle &\leq 0.\\
\end{align*}
Thus $p_\la^\star$ is the orthogonal projection of $\frac{y_0}{\la}$ on $C$: $p_\la^\star=P_{C}\left(\frac{y_0}{\la}\right)=p_\la$.
Since this is true for any subsequence, the whole sequence $p_{\la}^{\Gg_n}$ weakly converges to $p_\la$.

Moreover, by lower semincontinuity and the inclusion $C\subset C_n$ we have:
\begin{align*}
\left\|\frac{y_0}{\lambda}-p_\la\right\|_2 &\leq \liminf_{n\to +\infty} \left\|\frac{y_0}{\lambda}-p_{\la}^{\Gg_n}\right\|_2\leq \limsup_{n\to +\infty}\left\|\frac{y_0}{\lambda}-p_{\la}^{\Gg_n}\right\|_2 \leq \left\|\frac{y_0}{\lambda}-p_{\la}\right\|_2, \\
\end{align*}
so that $\frac{y_0}{\lambda}-p_{\la}^{\Gg_n}$ converges strongly to $\frac{y_0}{\lambda}-p_{\la}$,
 hence the strong convergence of $p_{\la}^{\Gg_n}$ to $p_\la$.

The rest of the statement follows from Proposition~\ref{prop-gamma-convergence}.

\bibliographystyle{plain}
\bibliography{bibliography}
  
\end{document}